\documentclass[12pt]{amsart}
\usepackage{amssymb, amscd, latexsym, graphicx, psfrag}
\usepackage[all]{xy}

\newcommand{\bC}{\mathbb{C}}
\newcommand{\bD}{\mathbb{D}}

\newcommand{\bF}{\mathbb{F}}

\newcommand{\bP}{\mathbb{P}}
\newcommand{\bQ}{\mathbb{Q}}
\newcommand{\bR}{\mathbb{R}}

\newcommand{\bZ}{\mathbb{Z}}

\newcommand{\cK}{\mathcal{K}}

\newcommand{\cO}{\mathcal{O}}
\newcommand{\cP}{\mathcal{P}}

\newcommand{\cT}{\mathcal{T}}
\newcommand{\cU}{\mathcal{U}}

\newcommand{\bfG}{\mathbf{G}}

\newcommand{\rmH}{\mathrm{H}}

\newtheorem{dummy}{dummy}[section]
\newtheorem{lemma}[dummy]{Lemma}
\newtheorem{theorem}[dummy]{Theorem}
\newtheorem{conjecture}[dummy]{Conjecture}

\newtheorem{proposition}[dummy]{Proposition}
\theoremstyle{definition}
\newtheorem{definition}[dummy]{Definition}
\newtheorem{example}[dummy]{Example}
\newtheorem{remark}[dummy]{Remark}

\newcommand{\Sp}{\mathrm{Sp}}
\newcommand{\SO}{\mathrm{SO}}

\newcommand{\U}{\mathrm{U}}
\newcommand{\Spin}{\mathrm{Spin}}
\newcommand{\SL}{\mathrm{SL}}
\newcommand{\G}{\mathrm{G}}
\newcommand{\PGL}{\mathrm{PGL}}
\newcommand{\F}{\mathrm{F}}
\newcommand{\E}{\mathrm{E}}

\newcommand{\Fun}{\mathrm{Fun}}
\newcommand{\Cor}{\mathcal{C}\mathrm{or}}

\newcommand{\Gm}{\mathbb{G}_\mathrm{m}}
\newcommand{\Ga}{\mathbb{G}_\mathrm{a}}

\newcommand{\Rc}{\mathbb{R}-c}
\newcommand{\Cc}{\mathbb{C}-c}
\newcommand{\FT}{\mathcal{FT}}

\renewcommand{\k}{\mathbf{k}}

\newcommand{\Sm}{\mathbf{Psm}}
\newcommand{\Bor}{\mathbf{Bor}}

\newcommand{\Perf}{\mathrm{Perf}}
\newcommand{\tate}{\mathcal{T}}

\newcommand{\SHA}{\mathrm{SHA}}
\newcommand{\Sat}{\mathrm{Sat}}
\newcommand{\fd}{\mathrm{fd}}

\newcommand{\Rep}{\mathrm{Rep}}

\renewcommand{\SS}{\mathit{SS}}

\begin{document}

\title{Smith theory and geometric Hecke algebras}
\author{David Treumann}
\date{July 19, 2011}

\begin{abstract}
In 1960 Borel proved a ``localization'' result relating the
rational cohomology of a topological space $X$ to the rational
cohomology of the fixed points for a torus action on $X$.  This result
and its generalizations have many applications in Lie theory.  In
1934, P. Smith proved a similar localization result relating the mod $p$
cohomology of $X$ to the mod $p$ cohomology of the fixed points for a
$\bZ/p$-action on $X$.  In this paper we study $\bZ/p$-localization for constructible
sheaves and functions.  We show that $\bZ/p$-localization on loop groups is related
via the geometric Satake correspondence to some special homomorphisms
that exist between algebraic groups defined over a field of small characteristic.
\end{abstract}

\maketitle

\tableofcontents

\section{Introduction}

It is often possible to compare the cohomology of a space $X$ to the cohomology of the fixed points for a torus action on $X$, by a technique than can be called ``equivariant localization.''  The prototype for this family of results goes back to Borel and Leray, e.g.

\begin{theorem}[Borel 1960, {\cite[Chapter XII.6]{Bor}}]
\label{thm:1.1}
Let $X$ be a finite-dimensional space with a $\U(1)$-action.  If $H^i(X;\bQ) = 0$ for $i$ odd, then $H^i(X^{\U(1)};\bQ) = 0$ for $i$ odd as well, and
$$\sum \dim(H^i(X;\bQ)) = \sum \dim(H^i(X^{\U(1)};\bQ))$$
\end{theorem}

For instance, if $G$ is a Lie group and $T$ is a maximal torus, the Theorem can be used to compute the rational cohomology of $G/T$ (or at least, its total rank).  Descendants of Borel's result (for instance \cite{AB} and \cite{GKM}) are used constantly in the deeper study of these and other spaces attached to a Lie group.

This paper concerns an important antecedent of Borel's result due to P. Smith:

\begin{theorem}[Smith 1934, \cite{S}]
\label{thm:1.2}
Let $p$ be a prime number and let $X$ be a finite-dimensional space with an action of $\bZ/p$.  If $H^i(X;\bZ/p) = 0$ for $i > 0$, then $H^i(X^{\bZ/p};\bZ/p) = 0$ for $i > 0$ as well, and
$$H^0(X^{\bZ/p};\bZ/p) = H^0(X;\bZ/p)$$
\end{theorem}

Theorem \ref{thm:1.2} appears weaker than Theorem \ref{thm:1.1}, but it is an essential ingredient in Borel's original proof (one may reduce to Theorem \ref{thm:1.2} by considering a large cyclic subgroup $\bZ/p \subset \U(1)$).  Borel's result has had an enormous impact on Lie theory.  One might hope that Smith's result would have a comparable impact on characteristic $p$ aspects of Lie theory, but to my knowledge this has not been the case.  Theorem \ref{thm:1.2} has been influential in algebraic topology and in cohomology of groups (e.g. \cite{Q}) but I do not know of many applications to representation theory.

The bread and butter of geometric representation theory is sheaf theory.  In this paper we develop ``Smith theory for sheaves'' and give an application to algebraic groups in small characteristic.  The main result of the formalism is the following---we will give only a rough version here in the Introduction.  Note that by ``sheaf'' in this paper we usually mean ``chain complex of sheaves'' or something similar.

\begin{theorem}
\label{thm:main}
If $X$ is a space with a $\bZ/p$-action, and $F$ is a $\bZ/p$-equivariant sheaf on $X$ defined over a field of characteristic $p$, then we may associate in a functorial way a sheaf $\Sm(F)$ on the fixed point set $X^{\bZ/p}$.  This assignment $F \mapsto \Sm(F)$, called the \emph{Smith operation}, commutes with all other sheaf operations, including duality, pushforward and pullback, nearby cycles, and microlocal stalks.  An equivariant version of $\Sm$ carries $G$-equivariant sheaves on $X$ to $Z_G(g)$-equivariant sheaves on $X^g$ whenever $g$ is an element of order $p$ in $G$.
\end{theorem}

The catch of the Theorem is that, while $F$ might be an ordinary sheaf (or complex of sheaves) of $\bF_p$-vector spaces, the sheaf $\Sm(F)$ is not defined over $\bF_p$ but over a certain $E_\infty$-ring spectrum $\tate$, called the ``Tate spectrum.''  We do not work directly with $\tate$ in this paper (and in particular the reader does not have to know what an $E_\infty$ ring spectrum is), but it's category of modules, which is easy to describe: it is the Verdier quotient of the category of bounded complexes of finitely generated $\bF_p[\bZ/p]$-modules by the category of bounded complexes of finitely generated free $\bF_p[\bZ/p]$-modules.   

To explain why the catch arises it is useful to consider what happens at the level of Grothendieck groups, i.e. to develop a ``Smith theory for functions.''  Let $f$ be a constructible function on $X$---for instance, $f$ might arise from a sheaf $F$ of $\bF_p$-modules by setting $f(x)$ equal to the Euler characteristic of the stalk of $F$ at $x$.  If $\Sm$ is to commute with all operations, then it commutes with taking stalks, giving us
\begin{equation}
\label{eq:smfx}
\Sm(f)(x) = f(x)
\end{equation}
whenever $x$ is a fixed point.  It should furthermore commute with taking global sections which at the level of functions is ``integration with respect to Euler measure'' (see Definition \ref{def:eulerint}), giving us
\begin{equation}
\label{eq:intXf}
\int_X f = \int_{X^{\bZ/p}} f
\end{equation}
However, Equation \ref{eq:intXf} only holds if we reduce both sides mod $p$, in which case it's a consequence of the fact that the Euler characteristic of a space with a free $\bZ/p$-action is divisible by $p$.

In other words, to get a good Smith theory at the level of functions we have to reduce the values of those functions mod $p$.  Working with the funny coefficients $\tate$ is analogous to this reduction in the following sense: while the Grothendieck group of the category of $\bF_p$-vector spaces  is $\bZ$, the Grothendieck group of the category of $\tate$-modules is $\bZ/p$.  The projection $\bZ \to \bZ/p$ is realized by a natural ``change of coefficients'' functor from $\bF_p$-vector spaces to $\cT$-modules.

\subsection{Hecke algebras and the Satake isomorphism}

Let $G$ be a Lie group acting on a space $X$.  In reasonable situations (for instance if $X$ and the action are real subanalytic) we may attach a ring to $X$ and $G$, called the Hecke algebra of the action.  The additive structure of the algebra is the group of $G$-invariant functions on $X \times X$, and the multiplication is given by
$$f_1 \ast f_2 (x,y) = \int f_1(x,z) \cdot f_2(z,y) dz$$
where the right-hand side again denotes the Euler characteristic integral.  A similar construction makes sense on a categorical level, yielding monoidal ``Hecke  categories'' that we discuss in Section \ref{sec:heckecat}.

We give an application of Smith theory by considering the spherical Hecke algebra of Satake, or rather the loop-group analog considered by Beilinson-Drinfeld and others.  Let $G$ be a compact group, $LG$ the space of free loops $\alpha:S^1 \to G$, and $\Omega G$ the space of based loops $\beta$ with $\beta(1) = 1$. 

\begin{remark}
As these are infinite-dimensional spaces, we will have to take some care to construct Hecke algebras.  In particular we will work with the usual affine Grassmannian model of $\Omega G$, see Section \ref{sec:ind} for details.
\end{remark}

The spherical Hecke algebra attached to $G$ is a group of $LG$-invariant functions on $\Omega G \times \Omega G$, where $LG$ acts on $\Omega G$ by
$$(\alpha \cdot \beta)(t) = \alpha(t) \beta(t) \alpha(1)^{-1}$$

The Satake isomorphism is an identification
$$\SHA_G \cong \Rep(G^\vee)$$
where the right hand side denotes the Grothendieck ring of representations of Langlands dual group $G^\vee$ to $G$.  

An element $g \in G$ acts by conjugation on $LG$ and $\Omega G$, and the fixed points are $L(Z_G(g))$ and $\Omega(Z_G(g))$---the free and based loop spaces of the centralizer of $g$.  In case $g$ has prime order $p$, Smith theory for functions provides a map
$$\Sm:\SHA_G \otimes_\bZ \bF_p \to \SHA_{Z_G(g)} \otimes_\bZ \bF_p$$
This map simply restricts an $\bF_p$-valued function on $\Omega G \times \Omega G$ to the fixed points $\Omega Z_G(g) \times \Omega Z_G(g)$.  Since this operation commutes with the Euler integral, it is a homomorphism of algebras.

It is natural to search for a representation-theoretic meaning of the corresponding homomorphism
$$\Sm:\Rep(G^\vee) \otimes_{\bZ} \bF_p \to \Rep(Z_G(g)^\vee) \otimes_{\bZ} \bF_p$$
The group $Z_G(g)^\vee$ is called an ``endoscopic group'' for $G^\vee$.  It often happens that $Z_G(g)$ is a Levi subgroup of $G^\vee$, in which case $Z_G(g)^\vee$ is a Levi subgroup of $G^\vee$ and the map $\Sm$ is given by restriction of representations.  However the endoscopic group is not in general a subgroup---for instance we may realize $Z = \Sp(2n) \times \Sp(2m)$ as the centralizer of an element of order 2 in $G = \Sp(2n+2m)$, but there is no way to include the subgroup $\SO(2n+1) \times \SO(2m+1)$ into $\SO(2n+2m+1)$.

We make two observations:
\begin{enumerate}
\item If we regard $G^\vee$ as an algebraic group over the field $K$, the representation ring $\Rep(G^\vee)$ is not sensitive to $K$.
\item There \emph{is} an inclusion of $\SO(2n+1)\times \SO(2m+1)$ into $\SO(2n+2m+1)$ so long as $K$ has characteristic 2.  See Section \ref{subsubsec:Cn}.
\end{enumerate}
More generally we prove the following Theorem:

\begin{theorem}
\label{thm:1.4}
Let $G$ be a compact simply connected simple Lie group, and let $g \in G$ be an element of order $p$ whose centralizer $Z_G(g)$ is semisimple.  Then the endoscopic group $Z_G(g)^\vee$ injects into $G^\vee$ when these groups are regarded as algebraic groups of characteristic $p$.  The restriction homomorphism $\Rep(G^\vee) \to \Rep(Z_G(g)^\vee)$ and the Smith homomorphism $\SHA_G \otimes_\bZ \bF_p \to \SHA_{Z_G(g)} \otimes_\bZ \bF_p$ are compatible with each other under Satake.
\end{theorem}

\begin{remark}
In particular, there is a natural lift of the Smith operator $\SHA_G \otimes_{\bZ} \bF_p \to \SHA_{Z_G(g)} \otimes_\bZ \bF_p$ to $\SHA_G \to \SHA_{Z_G(g)}$.  The Smith homomorphism does nothing more than restrict an $\bF_p$-valued function on $\Omega G \times \Omega G$ to the subset $\Omega Z_G(g) \times \Omega Z_G(g)$, but the lift to $\bZ$-valued functions is necessarily more intricate.  I do not know how to interpret this lift geometrically.
\end{remark}

\begin{remark}
It's natural to place this result in the context of results of Borel-de Siebenthal and Kac \cite{bds,kac}.  Let $G$ be a connected complex reductive group.  Then 
\begin{enumerate}
\item If $H \subset G$ is a connected subgroup of the same rank of $G$, then $H = Z_G(Z(H))$, i.e. $H$ is the centralizer of its center.
\item If $H \subset G$ is furthermore maximal among subgroups of full rank, then $H$ is the centralizer of an element of prime order.
\item Up to conjugacy, the elements of $G$ of order $k$ whose centralizer is semisimple are in one-to-one correspondence with simple roots whose coefficient in the maximal root of $G$ is equal to $k$.  
\end{enumerate}
Results (1) and (2) do not hold in characterstics 2 and 3.  In those characteristics maximal semisimple subgroups have been classified by Liebeck and Seitz \cite{ting}.  According to this classification, every such subgroup arises in the manner of Theorem \ref{thm:1.4}.  The converse is almost true---with the exception of a single conjugacy class of order 2 in $\F_4$, whenever $Z_G(g) \subset G$ is maximal, the subgroup $Z_G(g)^\vee \subset G^\vee$ is also maximal.
\end{remark}

\subsection{Hecke categories and the geometric Satake correspondence of Mirkovi{\'c}-Vilonen}
\label{sec:heckecat}

In the end, our proof of Theorem \ref{thm:1.4} is a case-by-case analysis.  But let us speculate about an alternative ``Tannakian'' proof.  Let $R$ be a commutative ring and let $\Sat(G,R)$ denote the ``$R$-linear Satake category''---this is the triangulated category of suitably constructible $LG$-equivariant sheaves of $R$-modules on $\Omega G \times \Omega G$.  $\Sat(G,R)$ is a categorification of the spherical Hecke algebra, with a monoidal structure that lifts the algebra structure on $\SHA_G$.  There is a second ``fusion'' product on $\Sat(G,R)$ as well, exhibiting it as a symmetric monoidal category.  A theorem of Mirkovi{\'c} and Vilonen \cite{MV} (following Lusztig, Beilinson-Drinfeld, and Ginzburg in case $R = \bC$) identifies the abelian subcategory $\cP(G,R) \subset \Sat(G,R)$ of perverse sheaves with the tensor category of representations of the split $R$-form of $G^\vee$.

\begin{remark}
We have called $\Sat(G,R)$ a ``symmetric monoidal category,'' but that is somewhat misleading.  We are regarding $\Sat(G,R)$ as a triangulated category in the sense of Verdier, but the more natural object is a certain stable $\infty$-category whose homotopy category is $\Sat(G,R)$.  In the $\infty$-categorical world, there is a hierarchy of commutativity constraints $E_2,E_3,\ldots,E_\infty$ on monoidal structures, and the one usually considered on $\Sat(G,R)$ is only $E_3$, not $E_\infty$, at the stable $\infty$-level.  The difference between $E_3$ and $E_\infty$ monoidal structures vanishes when we restrict attention to the subcategory $\cP$.
\end{remark}

There is a good version of the monoidal category $\Sat(G,R)$ when $R = \tate$ as well.  The categorical version of our Smith theory, Theorem \ref{thm:main}, gives us a functor
$$\Sm:\Sat(G,\bF_p) \to \Sat(Z_G(g),\tate)$$
As $\Sm$ commutes with all operations, including those used to define the convolution and fusion products on $\Sat(G,\bF_p)$ and $\Sat(Z_G(g),\tate)$, this can be shown to be a symmetric (or better, per the Remark, an $E_3$) monoidal functor.

There is no perverse $t$-structure on $\Sat(Z_G(g),\tate)$.  In fact, there can be no $t$-structure on $\Sat(Z_G(g),\tate)$ at all, for in the category of $\tate$-modules the identity functor is isomorphic to the shift-by-two functor.  However, we can ``extend coefficients'' from $\bF_p$ to $\cT$, which gives us a functor
$$\otimes_{\bF_p} \cT: \cP(Z_G(g);\bF_p) \to \Sat(Z_G(g);\tate)$$
As $\cP(Z_G(g);\bF_p)$ is equivalent to the category of representations of the $\bF_p$-form of $Z_G(g)^\vee$.  A consequence of Theorem \ref{thm:1.4} is that we may fill in the dotted arrow in the diagram
$$\xymatrix{
& \cP(Z_G(g);\bF_p) \ar[d]^{\otimes_{\bF_p} \cT} \\
\cP(G;\bF_p) \ar[r]_{\Sm\quad} \ar@{-->}[ur] & \Sat(Z_G(g);\cT)
}$$
Under the equivalence of Mirkovi{\'c} and Vilonen, the dotted arrow corresponds to the restriction functor of the inclusion $Z_G(g)^\vee \to G^\vee$ that exists over $\bF_p$.  With a better understanding of how $\Sm$ interacts with the theory of perverse sheaves, one might be able to give a Smith-theoretic construction of the dotted arrow, and presumably standard Tannakian considerations could then be used to deduce Theorem \ref{thm:1.4}.  We make a conjecture in Section \ref{sec:conjecture} along these lines.

\subsection{Notation and conventions}

To get a good theory of constructible functions and sheaves, we will work with real subanalytic sets.  We refer to \cite{KS} for the basic theory of subanalytic geometry.  We also often work with complex algebraic varieties, which we regard in their usual, locally compact and Hausdorff topology.  We will let $\chi_c$ denote the compactly supported Euler characteristic of a subanalytic or complex algebraic set, which is always well-defined.

The symbol $p$ always denotes a prime number.  We let $\k$ be a commutative ring (usually $\bZ$ or $\bZ/p$), and $K$ an algebraically closed field (usually $\overline{\bF_p}$).

If $G$ is a group, we write $Z_G(g)$ for the centralizer of an element $g \in G$.  The symbol $\varpi$ will usually denote a group of order $p$, and we sometimes write $Z_G(\varpi)$ instead of $Z_G(g)$ if $g$ has order $p$.

\section{Constructible functions and Hecke algebras}

Let $X$ be a real subanalytic set.  A function $f:X \to \k$ is called constructible if it is constant along the strata of a real subanalytic stratification of $X$, and zero away from a finite union of strata.  Write $\Fun_{\Rc}(X,\k)$ for the module of $\k$-valued functions on $X$ that are constructible with respect to a real subanalytic stratification of $X$.  If $X$ has a complex algebraic structure write $\Fun_{\Cc}(X,\k) \subset \Fun_{\Rc}(X,\k)$ for the $\k$-valued functions that are constructible with respect to a complex algebraic stratification of $X$.  When it is clear from context whether we are working in the real or complex setting, and which ring $\k$ we are considering, we will often write simply $\Fun(X)$.

\begin{definition}
\label{def:eulerint}
Let $\int:\Fun_{\Rc}(X,\k) \to \k$ denote the operator that takes a function $f$ to 
$$\sum_{i \in \k} \chi_c(f^{-1}(i)) \cdot i$$
Here $\chi_c$ denotes compactly-supported Euler characteristic.  
\end{definition}

\begin{remark}
Note that if $f$ is in $\Fun_{\Cc}(X,\k)$, then we have the alternative formula
$$\int f = \sum_{i \in \k} \chi(f^{-1}(i)) \cdot i$$
because of the relation $H^i_c(X) \cong H^{2n-i}(X)$ when $X$ is a smooth affine variety.
\end{remark}

If $u:X \to Y$ is a subanalytic (resp. complex algebraic) map, we define operations $u^*:\Fun(Y) \to \Fun(X)$ and $u_!:\Fun(X) \to \Fun(Y)$ by the formulas
$$
\begin{array}{rcl}
u^*(f)(x) & = &  f(u(x)) \\
u_!(f)(x) & = & \int_{u^{-1}(x)} f\vert_{u^{-1}(x)}
\end{array}
$$

We have, essentially by definition, the fundamental ``base-change'' relation:

\begin{proposition}
Suppose that the square
$$
\xymatrix{
X \ar[r]^v \ar[d]_{u'} & Y \ar[d]^u \\
Z \ar[r]_{v'} & W
}
$$
is Cartesian.  Then the operators $(v')^* \circ u_!:\Fun(Y) \to \Fun(Z)$ and $u'_! \circ v^*:\Fun(Y) \to \Fun(Z)$ are the same. 
\end{proposition}

\begin{remark}
\label{rem:cor1}
We can reformulate this relation in the language of categories.  Let $\Cor_\bR$ (resp. $\Cor_\bC$) denote the category whose objects are real subanalytic spaces and whose morphisms are ``correspondences'' $u = (\overleftarrow{u},\overrightarrow{u})$, i.e. diagrams of the form
$$X \stackrel{\overleftarrow{u}}{\leftarrow} U \stackrel{\overrightarrow{u}}{\rightarrow} Y$$
The composite of $u:X \to Y$ with $v:Y \to Z$ is given by the diagram
$$X \stackrel{\overleftarrow{u}\circ \mathrm{proj}}{\longleftarrow} U \times_{\overrightarrow{u},\overleftarrow{v}} V \stackrel{\overrightarrow{v} \circ \mathrm{proj}}{\longrightarrow} Z$$

A correspondence $u = (\overleftarrow{u},\overrightarrow{u})$ determines an operation $\overrightarrow{u}_! \overleftarrow{u}^*:\Fun(X) \to\Fun(Y)$.  The base-change relation is equivalent to the statement that $X \mapsto \Fun(X)$, $u \mapsto \overrightarrow{u}_! \overleftarrow{u}^*$ is a functor from $\Cor$ to the category of $\k$-modules.
\end{remark}

\subsection{Hecke algebras}
\label{sec:heckealgebras}

Let $G$ be a Lie group (resp. complex algebraic group) and suppose that it acts subanalytically (resp. algebraically) on $X$.  We define $\Fun_G(X)$ to be the $\k$-submodule of $\Fun(X)$ consisting of constructible functions that are constant along $G$-orbits.  If $u:X \to Y$ is a $G$-equivariant map then the operations $u^*$ and $u_!$ carry the $G$-invariant submodules into each other.  

We may use these operations to define a natural ring structure on $\Fun_G(X \times X)$, where $G$ acts by the diagonal action.  Specifically, given $f_1$ and $f_2$ in $\Fun_G(X \times X)$, we define
$$(f_1 \ast f_2)(x,y) = \int_z f_1(x,z) f_2(z,y)$$
It is clear that this new function is $G$-invariant.

\begin{definition}
The \emph{Hecke algebra} associated to a $G$-space is the algebra $\Fun_G(X \times X)$ just described.
\end{definition}

\begin{example}
\label{ex:doublecosets}
Let $X$ be a homogeneous space for $G$, say $X = G/H$.  Let $H \times H$ act on $G$ by $(h_1,h_2) \cdot g = h_1 g h_2^{-1}$.   The map $\Fun_G(X \times X) \to \Fun_{H \times H}(G)$ given by sending the function $f(g_1H,g_2H)$ to the function $f(1H,gH)$ is an isomorphism.  This identifies $\Fun_G(X \times X)$ with the group of functions on $G$ that are constant on $H-H$ cosets.  In particular, suppose $G$ is finite and $H$ is the trivial subgroup.  Then $\Fun_G(G \times G;\k)$ is naturally equivalent to the group ring of $G$ (though note that there are two such natural equivalences, which are exchanged by the involution $g \mapsto g^{-1}$ of the group ring).  
\end{example}

\subsection{The Smith operator}

Suppose now that $\k$ has characteristic $p > 0$ and let $\varpi$ be the group $\bZ/p$.  Let $X$ be a space with a $\varpi$ action, and let $X^\varpi$ denote the fixed points.  We define $\Fun_\varpi(X) \subset \Fun(X)$ to be the submodule of maps that are constant on $\varpi$-orbits.  The Smith operator is the map
$$\Sm:\Fun_\varpi(X) \to \Fun(X^\varpi):f \mapsto f\vert_{X^\varpi}$$

\begin{remark}
It is worth emphasizing that on functions $\Sm$ does nothing more than restrict $f:X \to \bF_p$ to the set of $\varpi$-fixed points.
\end{remark}

Suppose $u:X \to Y$ is a $\varpi$-equivariant map, and by abuse of notation write $u$ also for the map $X^\varpi \to Y^\varpi$.  We clearly have $u^*(\Sm(f)) = \Sm(u^*(f))$.  More interestingly, since $\k$ has characteristic $p$, we also have $u_!(\Sm(f)) = \Sm(u_!(f))$.

\begin{proposition}
\label{prop:Smith}
Let $X$ and $Y$ be $\varpi = \bZ/p$-spaces and let $u:X \to Y$ be a $\varpi$-equivariant map.  Let $X^\varpi$ and $Y^\varpi$ denote the fixed points.  Suppose that $\k$ has characteristic $p$.  Then the square
$$
\xymatrix{
\Fun_\varpi(X,\k) \ar[r]^{\Sm} 
\ar[d]_{u_!} 
& \Fun(X^\varpi,\k) 
\ar[d]^{u_!} \\
\Fun_\varpi(Y,\k) \ar[r]_{\Sm} & \Fun(Y^\varpi,\k)
}$$
is commutative
\end{proposition}

\begin{proof}
By the base-change relation, we may assume $Y$ is a point, i.e. we only have to show that
$$\int_{X^\varpi} \Sm(f) = \int_X f$$
when $\k$ has characteristic $p$.  By definition it suffices to show that 
$$\chi_c(f^{-1}(t)) - \chi_c(f^{-1}(t)^\varpi) = 0 \text{ mod }p$$
for every $t$.  This follows from the fact that each $f^{-1}(t) - f^{-1}(t)^\varpi$ is triangulable compatible with the free $\varpi$-action.
\end{proof}

\begin{remark}
\label{rem:cor2}
Let $\Cor_\varpi$ denote the category of $\varpi$-spaces and $\varpi$-equivariant correspondences, defined in the evident way.  The compatibility of the Smith operator with the operators $u_!$ and $u^*$ is equivalent to the statement that $\Sm$ defines a natural transformation between the functors of Remark \ref{rem:cor1}
$$
\xymatrix{
\Cor_\varpi \ar[rr]^{X \mapsto X^\varpi} \ar[dr]_{\Fun_\varpi} & & \Cor \ar[dl]^{\Fun} \\
& \k\text{-mod} &
}
$$
\end{remark}

\begin{remark}
\label{rem:pgroup}
We may generalize the Smith operator and Proposition \ref{prop:Smith} to the case of finite $p$-groups.  Let $\varrho$ be a $p$-group with $p^n$ elements and let $\Sm_{\varrho}$ denote the restriction map $\Fun_{\varrho}(X;\k) \to \Fun(X^\varrho;k)$.  We may always find a normal subgroup $\varrho' \subset \varrho$ of order $p^{n-1}$, in which case $\Sm_\varrho$ factors as
$$\Fun_\varrho(X;\k) \stackrel{\Sm_{\varrho'}}{\longrightarrow} \Fun_{\varrho/\varrho'}(X^{\varrho'};\k) \stackrel{\Sm_{\varrho/\varrho'}}{\longrightarrow} \Fun(X^\varrho;\k)$$.  By induction on $n$, we conclude that 
$$\int \Sm_{\varrho} f = \int f$$
when $f \in \Fun_{\varrho}(X;\k)$.
\end{remark}

\subsection{The Borel operator: localization for torus actions}
\label{sec:u1}

We can cast some torus-localization results in similar terms, and use the Smith operator to deduce them.  Let $T$ be a group isomorphic to $\U(1)^{\times n}$ if we are working in the real subanalytic setting or $(\bC^*)^{\times n}$ if we are working in the complex algebraic setting.  
Suppose that $T$ acts subanalytically or complex algebraically on $X$.  We have a restriction map
$$\Fun_T(X;\bZ) \to \Fun(X^T;\bZ): f \mapsto f\vert_{X^T}$$
We refer to this map as the \emph{Borel operator} $\Bor$.

\begin{remark}
\label{rem:BorSmcompat}
We have a basic compatibility between the Borel and Smith operators.  For each $p$, we may find a  $p$-group $\varrho$ contained in $T$ with $X^\varrho = X^T$.  In that case it is clear that reducing the Borel map for $T$ mod $p$ yields the Smith map for $\varrho$.
\end{remark}

\begin{proposition}
Suppose $T$ acts on both $X$ and $Y$ and that $f:X \to Y$ is $T$-equivariant.  The square
$$
\xymatrix{
\Fun_T(X,\bZ) \ar[r]^{\mathbf{Bor}} 
\ar[d]_{u_!} 
& \Fun(X^T,\bZ) 
\ar[d]^{u_!} \\
\Fun_T(Y,\bZ) \ar[r]_{\mathbf{Bor}} & \Fun(Y^T,\bZ)
}$$
commutes.
\end{proposition}

\begin{proof}
Let us give the proof in the subanalytic setting, the complex algebraic version is similar.  We may reduce to the case $T = \U(1)$ by induction on the dimension of $T$.  For each prime $p$ we have an inclusion $\bZ/p \subset \U(1)$, and since the action is subanalytic we have $X^{\U(1)} = X^{\bZ/p}$ for $p$ sufficiently large.  The Proposition now follows from Proposition \ref{prop:Smith} and Remark \ref{rem:BorSmcompat}, and the observation that a map between free abelian groups is determined by its reduction mod infinitely many primes.
\end{proof}

\subsection{The Smith operator for $G$-spaces}
Let $G$ be a Lie group (or a complex algebraic group) and let $X$ be a $G$-space.  If we fix a subgroup of $G$ of the form $\varpi = \bZ/p$, then $\Fun_G(X) \subset \Fun_\varpi(X)$.  Moreover, if $f \in \Fun_G(X)$ then $\Sm(f):X^\varpi \to \k$ is constant along $Z_G(\varpi)$-orbits, where $Z_G(\varpi)$ is the centralizer of $\varpi$ in $G$.  We define the equivariant Smith operator to be the function
$$\Sm:\Fun_G(X) \to \Fun_{Z(\varpi)}(X^\varpi)$$

\begin{remark}
\label{rem:normalizer}
In fact $\Sm(f)$ is constant along orbits for the slightly larger group $N_G(\varpi) \supset Z_G(\varpi)$, the normalizer of $\varpi$ in $G$.  If $\varpi$ is a larger $p$-group or torus, then $N_G(\varpi)$ can be quite a bit larger than $Z_G(\varpi)$ and this extra invariance is useful.  However at the level of sheaves this larger group turns out to play a different role than $Z_G(\varpi)$, so we prefer to regard $\Sm$ as taking values in $\Fun_{Z_G}(X)$.
\end{remark}

Combining the results of this section gives us the following:

\begin{theorem}
\label{thm:smithhecke}
Let $G$ act on $X$ and let $\varpi$ be a subgroup of $G$ of the from $\bZ/p$.  The Smith operator
$$\Sm:\Fun_G(X \times X,\k) \to \Fun_{Z_G(\varpi)}(X^\varpi \times X^\varpi,\k)$$
is a $\k$-algebra homomorphism.
\end{theorem}

\subsection{Ind-varieties}
\label{sec:ind}

We will explain how to extend some of the formalism of constructible functions, Hecke algebras, and the Smith operator to the setting of ind-complex varieties.  In this paper (in section \ref{sec:SHA}) we only treat the case of $G(\cK)$ acting on the affine Grassmannian, but we make a few general remarks here.

For us, an ind-variety is a topological space $X$ together with an exhaustive filtration $X \supset \cdots X_i \supset X_{i-1} \supset \cdots \supset X_0$ by closed subspaces, each equipped with the structure of a complex algebraic variety and each inclusion $X_j \hookrightarrow X_{j+1}$ being algebraic.  See \cite[Section 2.2]{Nadler} for an exposition of ind-varieties relevant for our applications.

Suppose $X = \bigcup X_i$ is an ind-variety.  Denote by $\Fun^{\fd}(X;\k)$ the $\k$-module of functions on $X$ that are supported on one of the $X_i$, and that are constructible there.  We call $\Fun^{\fd}(X;\k)$ the $\k$-module of constructible functions with finite-dimensional support.  On ind-varieties of the form $X \times X$ it is useful to consider a larger class of functions.  Let us say a function $f:X \times X \to \k$ has ``property H'' if for every finite-dimensional subvariety $Z \subset X$, the functions $f \vert_{Z \times X}$ and $f\vert_{X \times Z}$ have finite-dimensional support.  We $\Fun^{\rmH}$ denote the group of functions satisfying property $H$.  It is clear that the operation $(f,g) \mapsto f \ast g$ where
$$f \ast g(x,y) = \int_z f(x,z) g(z,y)$$
is well-defined whenever both $f$ and $g$ have property H.

\begin{example}
In case $X$ is discrete, we may identify $\Fun^{\rmH}(X \times X;\k)$ with the $\k$-module of infinite square matrices each row and each column of which have only finitely many nonzero entries.
\end{example}

\begin{example}
If $G$ is an infinite discrete group then we may identify $\Fun_G^{\mathrm{H}}(G \times G)$ with the group ring of $G$ as in Example \ref{ex:doublecosets}.
\end{example}

It is clear that whenever $Y$ is an ind-subvariety of $X$, a function $f:X \times X \to \k$ with property H restricts to a function on $Y \times Y$ with property $H$.  In particular if $\varpi$ acts on $X$ then we have a natural restriction map
$$\Sm:\Fun_\varpi^{\rmH}(X \times X;\k) \to \Fun^{\rmH}(X^\varpi \times X^\varpi;\k)$$
which is an algebra homomorphism if $\k$ has characteristic $p$ and which we call the ``Smith operator'' as usual.

It is natural to consider actions of ind-pro algebraic groups $\bfG$ on ind-varieties $X$.  In that case there is a subalgebra $\Fun_{\bfG}^{\rmH}(X\times X;\k) \subset \Fun^{\rmH}(X \times X;\k)$ of H-functions that are constant on the $\bfG$-orbits of $X \times X$.  If we have an inclusion $\varpi \subset \bfG$ then we have a ind pro algebraic subgroup $Z_\bfG(\varpi) \subset \bfG$ that centralizes $\varpi$.  In that case the Smith operator defines an algebra homomorphism

$$\Fun^{\rmH}_{\bfG}(X \times X;\k) \to \Fun^{\rmH}_{Z_{\bfG}(\varpi)}(X^\varpi \times X^\varpi;\k)$$

\subsection{Smith and the natural operators on constructible functions}

Let us recall some other important operations on functions, and show that the Smith operators is compatible with all of them.

\subsubsection{Duality}

Let $M$ be a real analytic manifold.  Let us call an open subset $U$ of $M$ \emph{good} if there exists a subanalytic triangulation of $M$ for which $U$ is the star of a vertex.  We moreover require that the triangulation is fine enough that the boundary of $U$ is a topological sphere.  In that case we define constructible functions $i_U$ and $j_U$ by
$$
\begin{array}{c}
i_U(x) := \bigg\{ 
   \begin{array}{ll}
   1 & \text{if $x$ is in the closure of $U$}\\
   0 & \text{otherwise}
   \end{array} \\
j_U(x) := \bigg\{
   \begin{array}{ll}
   (-1)^{\dim(U)} & \text{if $x$ is in $U$}\\
   0 & \text{if $x$ is outside of or on the boundary of $U$}
   \end{array}
\end{array}
$$
We refer to $i_U$ and $j_U$ as the \emph{standard} and \emph{costandard} functions on $M$ associated to $U$ respectively.  
Standard results on the existence of subanalytic triangulations imply that the functions $i_U$ (resp. $j_U$) generate $\Fun(M;\k)$ as a $\k$-module.  

\begin{remark}
The following variations on this fact are useful:
\begin{enumerate}
\item Suppose $\{\cU_i\}_{i \in I}$ is an open cover of $M$.  Then the collection of functions $j_U$ where $U$ is a good open subset of $M$ entirely contained in one of the $\cU_i$ span $\Fun(M;\k)$.
\item Suppose $X \subset M$ is a subanalytic set.  Then the functions of the form $j_U\vert_X$ span $\Fun(X;\k)$.
\end{enumerate}
\end{remark}

From this we can deduce the following:

\begin{theorem}
\label{thm:dualityfun}
There is a unique system of operators $\bD_X:\Fun(X) \to \Fun(X)$ satisfying the following conditions:
\begin{enumerate}
\item When $M$ is a real analytic manifold and $U \subset M$ is a subanalytic open subset with smooth boundary, then $\bD_X(i_U) = \bD_X(j_U)$.
\item When $f:X \to Y$ is a closed immersion, then $\bD_Y \circ f_! = f_! \circ \bD_X$.
\end{enumerate}
The operators $\bD_X$ satisfy the following additional properties:
\begin{enumerate}
\item[(3)] When $X \stackrel{\overleftarrow{u}}{\leftarrow} Z \stackrel{\overrightarrow{u}}{\rightarrow} Y$ is a subanalytic correspondence with $\overleftarrow{u}$ \'etale and $\overrightarrow{u}$ proper, the operator $\overrightarrow{u}_! \overleftarrow{u}^*:\Fun(X) \to \Fun(Y)$ intertwines $\bD_X$ and $\bD_Y$.
\item[(4)] $\bD_X \circ \bD_X = \mathrm{id}$
\end{enumerate}
\end{theorem}

\begin{proof}
See \cite[Theorem 2.5]{Schapira}
\end{proof}

\begin{definition}
Let $\bD_X$ be the operators of Theorem \ref{thm:dualityfun}.  Let $u:X \to Y$ be a subanalytic map between subanalytic varieties.  We introduce the following notation:
\begin{enumerate}
\item We refer to $\bD_X$ as the duality operator of $X$
\item We let $u_*:\Fun(X) \to \Fun(Y)$ denote the unique operator with $\bD_Y \circ u_* = u_! \circ \bD_X$.
\item We let $u^!:\Fun(Y) \to \Fun(X)$ denote the unique operator with $\bD_Y \circ u^! = u^* \circ \bD_Y$.
\end{enumerate}
\end{definition}

Thus, if $u$ is proper then $u_* = u_!$, and if $u$ is \'etale then $u^! = u^*$.  

If $\varpi$ acts subanalytically on $X$, then applying property (2) to the translation maps in $\varpi$ we see that $\bD_X$ must carry $\varpi$-invariant functions to $\varpi$-invariant functions.  Similarly $u_*$ and $u^!$ carry $\varpi$-invariant functions to $\varpi$-invariant functions whenever $u$ is $\varpi$-equivariant.

\begin{proposition}
\label{prop:Smithduality}
Let $X$ and $Y$ be subanalytic varieties equipped with $\varpi$-actions, and let $u$ be a $\varpi$-equivariant map.  Suppose $\k$ has characteristic $p$.
\begin{enumerate}
\item We have $\Sm \circ \bD_X = \bD_X \circ \Sm$ 
\item We have $u_* \circ \Sm = \Sm \circ u_*$
\item We have $u^! \circ \Sm = \Sm \circ u^!$

\end{enumerate}

\end{proposition}

\begin{proof}
Pick a subanalytic embedding $X \hookrightarrow V$ into a real vector space $V$, and suppose that the embedding is equivariant for a linear $\varpi$-action on $V$.  It suffices to verify that $\bD_{V^\varpi} \circ \Sm(f)  = \Sm \circ \bD_{V}(f)$ for all $f \in \Fun_\varpi(V;\k)$, and since the costandard functions $j_U$ where $U$ is $\varpi$-invariant good open subset generate we may assume $f$ is of this form.  Since the $\varpi$-action is linear, $U^\varpi$ is a good open subset of $V^\varpi$, so $j_U\vert_{V^\varpi} = j_{U^\varpi}$.  Then $\bD_{V^\varpi} \circ \Sm(j_U) = i_{U^\varpi} = \Sm(i_U) = \Sm \circ \bD_V(j_U)$.  This proves property (1).  

Properties (2) and (3) follow from property (1), Proposition \ref{prop:Smith}, and the definitions. 
\end{proof}

\subsubsection{Specialization}
\label{sec:specialization}

\begin{definition}
Let $X$ be a subanalytic variety and let $u:X \to \bR$ be a subanalytic map.  Let $i$ denote the inclusion $u^{-1}(0) \hookrightarrow X$ and $j$ the inclusion $u^{-1}\{ t \mid t > 0\} \hookrightarrow X$.  The \emph{upper specialization} operator is the homomorphism $\psi_u^+:\Fun(X;\k) \to \Fun(u^{-1}(0);\k)$ given by 
$$\psi_u^+(f) = i^* j_* j^* f$$
\end{definition}

If $\varpi$ acts on $X$ and $u$ is constant on $\varpi$-orbits, $\varpi$ also acts on $u^{-1}(0)$ and the specialization map carries $\Fun_\varpi(X;\k)$ to $\Fun_\varpi(u^{-1}(0);\varpi)$.

\begin{proposition}
\label{prop:speccompat}
Let $\varpi$ act on $X$ and let $\k$ have characteristic $p$.  Suppose $u:X \to \bR$ is constant on $\varpi$-orbits.  We have a commutative square
$$
\xymatrix{
\Fun_\varpi(X;\k) \ar[r]^{\Sm} \ar[d]_{\psi^+_u} & \Fun(X^\varpi,\k) \ar[d]^{\psi_{u\vert_{X^\varpi}}^+} \\
\Fun_\varpi(u^{-1}(0);\k) \ar[r]_{\Sm} & \Fun(u^{-1}(0)^\varpi,\k)
}
$$
\end{proposition}

\subsubsection{Fourier-Sato transform}
\label{sec:FT}

\begin{definition}
Let $V$ be a real vector space.  We say a constructible function $f:V \to \k$ is \emph{conic} if $f(t \cdot v) = f(v)$ for all real numbers $t > 0$.  Let $\Fun_{\bR_{>0}}(V;\k) \subset \Fun(V;\k)$ denote the $\k$-module of conical constructible functions.  The \emph{Fourier-Sato transform} $\FT:\Fun_{\bR_{>0}}(V;\k) \to \Fun_{\bR_{>0}}(V^*;\k)$ given by
$$\FT(f)(\xi) = \int_{\{v \in V \mid \xi(v) < 1\}} f$$
\end{definition}

\begin{remark}
We have an evident relative version of this notion: if $V$ is a vector bundle over a space $X$, then $\FT$ carries dilation-invariant functions on the total space of $V$ to dilation-invariant functions on the total space of $V^*$.
\end{remark}

Suppose now that $\varpi$ acts linearly on $V$.  Note that the restriction map $(V^*)^\varpi \to (V^\varpi)^*$ is an isomorphism; we compute the inverse by sending a functional $\xi:V^\varpi \to \bR$ to the $\varpi$-invariant functional $\tilde{\xi}$ given by
$$\tilde{\xi}(v) = \frac{1}{p} \xi(\sum_{g \in \varpi} gx)$$
Write $\Fun_{\bR_{>0},\varpi}(V;\k)$ for the $\k$-module of constructible functions that are conical and constant on $\varpi$-orbits.  The Smith operator defines a map
$$\Sm:\Fun_{\bR_{>0},\varpi}(V;\k) \to \Fun_{\bR_{>0}}(V^\varpi;\k)$$

\begin{proposition}
\label{prop:FTcompat}
Let $V$ be a real vector space with a linear $\varpi$-action, and endow the dual vector space $V^*$ with the contragredient $\varpi$-action.  The Smith operator is compatible with the Fourier-Sato transform and the identification $(V^\varpi)^* = (V^*)^\varpi$.  That is, the following diagram commutes
$$
\xymatrix{
\Fun_{\bR_{>0},\varpi}(V,\k) \ar[rr]^{\FT} \ar[d]_{\Sm} & & \Fun_{\bR_{>0},\varpi}(V^*;\k) \ar[d]^{\Sm} \\
\Fun_{\bR_{>0}}(V^\varpi;\k) \ar[r]^{\FT} & \Fun_{\bR_{>0}}((V^\varpi)^*;\k) \ar[r]^{\xi \mapsto \tilde{\xi}} & \Fun_{\bR_{>0}}((V^*)^\varpi;\k)
}
$$
\end{proposition}

\begin{proof}
As the averaging map $\Fun_{\bR_{>0}}((V^\varpi)^*;\k) \to \Fun_{\bR_{>0}}((V^*)^\varpi;\k)$ is inverse to the restriction map $\Fun_{\bR_{>0}}((V^*)^\varpi;\k) \to \Fun_{\bR_{>0}}((V^\varpi)^*;\k)$, the Proposition is equivalent to the assertion that the diagram
$$
\xymatrix{
\Fun_{\bR_{>0},\varpi}(V,\k) \ar[rr]^{\FT} \ar[d]_{\Sm} & & \Fun_{\bR_{>0},\varpi}(V^*;\k) \ar[d]^{\Sm} \\
\Fun_{\bR_{>0}}(V^\varpi;\k) \ar[r]^{\FT} & \Fun_{\bR_{>0}}((V^\varpi)^*;\k) & \Fun_{\bR_{>0}}((V^*)^\varpi;\k) \ar[l]
}
$$
commutes.  Let $f:V \to \k$ be a conical, $\varpi$-invariant constructible function on $V$.  The associated map $(V^\varpi)^* \to \k$ given by traveling through the upper right corner of the square is given by
$$\xi \mapsto \int_{\{v \in V \mid \xi(v) < 1\}} f$$
while going through the bottom left corner is given by
$$\xi \mapsto \int_{v \in V^\varpi \mid \xi(v) < 1\}} f\vert_{V^{\varpi}}$$
That these agree follows from the Proposition \ref{prop:Smith}.
\end{proof}

\subsubsection{Specialization and microlocalization}

Let $X$ be a manifold and $M \subset X$ a closed submanifold.  Let $T_M X$ denote the normal bundle and $T_M^* X$ the conormal bundle to $M$ in $X$.  Using the operators of \ref{sec:specialization} and \ref{sec:FT} one defines operators
$$
\begin{array}{c}
\nu_M:\Fun(X;\k) \to \Fun(T_M X;\k) \\
\mu_M:\Fun(X;\k) \to \Fun(T_M^* X;\k)
\end{array}
$$
called ``specialization along $M$'' and ``microlocalization along $M$'' respectively.  The definition involves a new manifold $\widetilde{X}_M$ (see \cite[Chapter IV]{KS} for a construction) called the ``deformation to the normal bundle'' of $M$ in $X$.  $\widetilde{X}_M$ is equipped with an action of $\bR_{>0}$ and a map $\pi:\widetilde{X}_M \to \bR$ with the following properties:

\begin{enumerate}
\item $\pi$ is $\bR_{>0}$-equivariant
\item $\pi^{-1}(t)$ is naturally identified with $X$ for $t \neq 0$
\item $\pi^{-1}(0)$ is naturally identified with the normal bundle $T_M X$.
\end{enumerate}

For each constructible function $f:X \to \k$ we may find a $\bR_{>0}$-invariant constructible function $f_1:\widetilde{X}_M \to \k$ whose restriction to each nonzero fiber agrees with $f$ under the identification (2).  Then we set $\nu_M(f) = \psi_\pi^+(f_1)$ and (as $\nu_M(f)$ is $\bR_{>0}$-invariant) $\mu_M(f) = \FT(\nu_M(f))$.

The construction $(X,M) \mapsto \widetilde{X}_M$ is functorial, in particular if $\varpi$ acts on $X$ and $M$ is stable for this action then $\varpi$ acts on $\widetilde{X}_M$.  Thus we have $\varpi$-equivariant versions of the operators $\nu_M$ and $\mu_M$.
$$\begin{array}{c}
\nu_M:\Fun_\varpi(X;\k) \to \Fun_{\bR_{>0},\varpi}(T_M X;\k)\\
\mu_M:\Fun_\varpi(X;\k) \to \Fun_{\bR_{>0},\varpi}(T^*_M X;\k)
\end{array}
$$

We can identify $(T_M X)^\varpi \cong T_{M^\varpi} X^\varpi$ and $(T_M X)^{*\varpi} \cong T_{M^\varpi}^* X^\varpi$ By Proposition \ref{prop:speccompat} and \ref{prop:FTcompat} we have

\begin{proposition}
Suppose $\varpi$ acts on a manifold $X$ and $M$ is a $\varpi$-invariant submanifold.  Then 
we have commutative squares 
$$
\xymatrix{
\Fun_\varpi(X;\k) \ar[r]^{\nu_M} \ar[d]_{\Sm}  & \Fun_\varpi(T_M X;\k) \ar[d]^{\Sm}  & 
\Fun_\varpi(X;\k) \ar[r]^{\mu_M} \ar[d]_{\Sm}&
\Fun_\varpi(T^*_M X;\k) \ar[d]^{\Sm} \\
\Fun(X^\varpi;k) \ar[r]_{\nu_{M^\varpi}} \ar[r] & \Fun(T_{M^\varpi} X^\varpi; \k)  & \Fun_\varpi(X^\varpi;\k) \ar[r]_{\mu_{M^\varpi}} & \Fun(T^*_{M^\varpi} X^\varpi;\k)
}
$$
\end{proposition}

\subsubsection{Singular support}
\label{sec:singsupp}

Let $M$ be a manifold and let $f:M \to \k$ be a constructible function.  We define a subset $\SS(f) \subset T^*M$ to be the closure of the set of $(x,\xi) \in T^*M$ with the following property: in every sufficiently small neighborhood $U$ of $x$ and $\epsilon >0$, if $\psi:U \to \bR$ is a smooth function with $d\psi_x = \xi$, then
$$\int_{\{u \in U \mid \psi(u) \leq \psi(x) + \epsilon\}} f \neq \int_{\{u \in U \mid \psi(u) \leq \psi(x) - \epsilon\}} f$$

If $\varpi$ acts on $M$ and $f$ is $\varpi$-invariant, then $\SS(f)$ is also $\varpi$-invariant.  We may identify $\SS(f)^\varpi$ with a subset of $T^* (M^\varpi)$, and we have the following:

\begin{proposition}
Suppose $\k$ has characteristic $p$.  Then we have $\SS(\Sm(f)) = \SS(f)^\varpi$.
\end{proposition}

\begin{proof}
The equality is local on $M$, so it suffices to consider the case where $M$ is an open subset of a vector space $V$.  In that case we identify $\SS(f)^\varpi$ with a subset of $T^*(V^\varpi)$ by sending $(x,\xi) \in \SS(f)^\varpi \subset V \times V^*$ to $(x,\xi\vert_{V^\varpi})$.

A pair $(x,\xi) \in V^\varpi \times (V^\varpi)^*$ belongs to $\SS(\Sm(f))$ if and only if in a neighborhood $U$ of $x$ we may find a smooth $\psi:U \to \bR$ such that $d\psi_x = f$ and
$$\int_{\{u \in U \mid \psi(u) \leq \psi(x) + \epsilon\}} \Sm(f) \neq \int_{\{u \in U \mid \psi(u) \leq \psi(x) - \epsilon\}} \Sm(f)$$
We may always extend $U$ to a $\varpi$-stable open subset $U_1 \subset V$ with $U_1^\varpi = U$, and we may define $\psi_1:U_1 \to \bR$ by
$$\psi_1(u) = \psi(\frac{1}{p}\sum_{g \in \varpi} gu)$$
Then $\psi_1 \vert_U = \psi$ and $\{u \in U \mid \psi(u) \leq t\} = \{u \in U_1 \mid \psi_1(u) \leq t\}^\varpi$, so that by Proposition \ref{prop:Smith} and the equation above we have  
$$\int_{\{u \in U_1 \mid \psi_1(u) \leq \psi_1(x) + \epsilon\}} f \neq \int_{\{u \in U_1 \mid \psi_1(u) \leq \psi_1(x) - \epsilon\}} f$$
As $(d\psi_1)_x = d\psi_x = \xi$, this is equivalent to $(x,\xi)$ belonging to $\SS(f)$.
\end{proof}

\section{Smith theory for the spherical Hecke algebra}
\label{sec:SHA}

In this section we work with complex algebraic varieties and ind-varieties.  Let $\cK$ denote the field of Laurent series $\bC((t))$ and $\cO$ denote the ring  of Taylor series $\bC[[t]]$.  If $G$ is a complex reductive algebraic group then we have an ind-group $G(\cK)$, a subgroup $G(\cO)$, and a coset space $G(\cK)/G(\cO)$ which is an ind-variety in a natural way.

\begin{definition}
Let $\k$ be a commutative ring and let $G$ be a complex connected reductive algebraic group.  The \emph{spherical Hecke algebra} is the $\k$-module of constructible functions $\Fun^{\rmH}_{G(\cK)}(G(\cK)/G(\cO) \times G(\cK)/G(\cO))$ endowed with the convolution product defined in Sections \ref{sec:heckealgebras} and \ref{sec:ind}.  We will denote the spherical Hecke algebra attached to $G$ and $\k$ by $\SHA_{G,\k}$.
\end{definition}

\begin{remark}
Note that if $G$ is not connected then $G(\cK)/G(\cO)$ is naturally identified with $G^\circ(\cK)/G^{\circ}(\cO)$, where $G^\circ$ is the neutral component of $G$.  In that case $\pi_0(G) = \pi_0(G(\cK))$ acts on $\SHA_{G^{\circ},\k}$ and the invariant subring is $\SHA_{G,\k}$.
\end{remark}

Suppose $G$ is connected.  The Satake isomorphism identifies $\SHA_{G,\k}$ with the representation ring $\Rep(G^\vee)$ of the Langlands dual group to $G^\vee$.  If $\varpi$ is a subgroup of order $p$ in $G$ and $\k$ has characteristic $p$ then it turns out the Smith homomorphism maps $\SHA_{G,\k}$ to $\SHA_{Z_G(\varpi);\k}$.  If $G$ is simply connected then $Z_G(\varpi)$ is connected, and we can ask whether the corresponding homomorphism
$$\Rep(G^\vee) \otimes_\bZ \k \to \Rep(Z_G(\varpi)^\vee) \otimes_{\bZ} \k$$
has an interpretation.  Before stating the main result of this section we make a couple of more observations:
\begin{itemize}
\item We can regard $G^\vee$ and $Z_G(\varpi)^\vee$ as algebraic groups defined over any algebraically closed field (or split algebraic groups defined over any field at all) without changing the structure of $\Rep(G^\vee)$ and $\Rep(Z_G(\varpi)^\vee)$.
\item As $\varpi$ is commutative, $Z_G(\varpi)$ contains a maximal torus of $G$ and therefore has the same rank as $G$.  It follows that $Z_G(\varpi)^\vee$ has the same rank as $G^\vee$.
\end{itemize}

\begin{theorem}
\label{thm:centralizer}
Let $G$ be a simply connected simple complex algebraic group.  Let $\varpi \subset G$ be a subgroup of order $p$ and suppose that the centralizer $Z_G(\varpi)$ of $\varpi$ in $G$ is semsimple.  Let $G^\vee$ and $Z_G(\varpi)^\vee$ be the Langlands dual groups over an algebraically closed field $K$
\begin{enumerate}
\item If $K$ has characteristic $p$, then there is an inclusion of $Z_G(\varpi)^\vee$ into $G^\vee$
\item The canonical bijections $\mathrm{coroots}(Z_G(\varpi)^\vee)) \cong \mathrm{roots}(Z_G(\varpi))$ and $\mathrm{coroots}(G^\vee) \cong \mathrm{roots}(G)$ commute with the inclusions induced by $Z_G(\varpi) \hookrightarrow G$ and $Z_G(\varpi)^\vee \hookrightarrow G^\vee$.
\item The square
$$\xymatrix{
\Rep(G^\vee) \otimes_\bZ \k \ar[r]^{\mathrm{res}\quad} \ar[d]_{\cong} & \Rep(Z_G(\varpi)^\vee) \otimes_{\bZ} \k \ar[d]^{\cong} \\
\SHA_{G,\k} \ar[r]_{\Sm} & \SHA_{Z_G(\varpi),\k}
}
$$
commutes.
\end{enumerate}

\end{theorem}

\begin{remark}
\label{rem:autom}
Assertion (2) is automatic from assertion (1): if there in an inclusion of $Z_G(\varpi)^\vee$ into $G^\vee$ then it must induce an inclusion of roots that preserves lengths and angles, and there is only one possibility.  Moreover according to Lemma \ref{lem:37} below, assertion (3) is implied by the following consequence of assertion (2): the induced map on Weyl groups $W_{Z_G(\varpi)^\vee} \to W_{G^\vee}$ is the same as the map $W_{Z_G(\varpi)} \hookrightarrow W_G$ induced by the inclusion of $Z_G(\varpi)$ into $G$, and the natural isomorphisms $W_{G^\vee} \cong W_G$, $W_{Z_G(\varpi)^\vee} \cong W_{Z_G(\varpi)}$.  To prove the Theorem we will therefore only have to check (1). 
\end{remark}

In Section \ref{sec:31} we prove the Satake isomorphism with the aid of the Smith and Borel operators.   In Section \ref{sec:33}, we verify Theorem \ref{thm:centralizer} by a case-by-case analysis.

\subsection{The Satake isomorphism via Smith theory}
\label{sec:31}

A subgroup $H$ of $G$ acts on $G(\cK)/G(\cO)$ by translation, and on  $G(\cK)$ by conjugation.  We have the following basic fixed-point calculation:

\begin{proposition}
Let $H$ be a reductive subgroup of $G$, and let $Z$ be the centralizer of $H$ in $G$.
\begin{itemize}
\item $G(\cK)^H = Z(\cK)$
\item $(G(\cK)/G(\cO))^H = Z(\cK)/Z(\cO)$
\end{itemize}
\end{proposition}

We will focus on the following special cases:

\begin{itemize}
\item If $H$ is a finite $p$-group and $\k$ has characteristic $p$, then we may consider the Smith map
$$\Sm:\SHA_{G,\k} \to \SHA_{Z_G(H),\k}$$
of Remark \ref{rem:pgroup}.
\item If $H$ is a connected torus we may consider the Borel map
$$\Bor:\SHA_{G,\bZ} \to \SHA_{Z_G(H),\k}$$
of Section \ref{sec:u1}.
\end{itemize}

Let us consider the Borel map first, in the case when $H = T$ is a maximal torus.  With it we may deduce a version of the classical Satake isomorphism.  

\begin{theorem}[Satake]
Let $G$ be a complex reductive algebraic group, let $T$ be a maximal torus, and let $W$ be the Weyl group.  If we identify $\SHA_{T,\bZ}$ with the group ring $\bZ[X_*(T)]$ by the method of Example \ref{ex:doublecosets}, then the Borel operator for $T$
$$\Bor:\SHA_{G,\bZ} \to \SHA_{T,\k} \cong \bZ[X_*(T)]$$  
is an isomorphism onto the ring of $W$-invariants $\bZ[X_*(T)]^W$.
\end{theorem}

\begin{proof}
The $G(\cO)$-orbits on $G(\cK)/G(\cO)$ are in one-to-one correspondence with the $W$-orbits on $X_*(T)$ by the map $O \mapsto O \cap X_*(T)$.  Equivalently, 
the $G(\cK)$-orbits on $G(\cK)/G(\cO) \times G(\cK)/G(\cO)$ are in one-to-one correspondence with the $X_*(T) \rtimes W$-orbits on $X_*(T) \times X_*(T)$.  It follows that restricting a function on $f \in \SHA_{G,\bZ}$ to a function on $X_*(T) \times X_*(T)$ is an isomorphism onto $W$-invariants.
\end{proof}

\subsubsection{Note on representation rings}
\label{sec:norr}

Let $K$ be an algebraically closed field and let $G$ be a reductive algebraic group defined over $K$.  Let $T \subset G$ be a maximal torus for $G$ and let $W$ be the Weyl group of $G$.  Let $\Rep(G)$ denote the representation ring of $G$, i.e. the Grothendieck ring of the $K$-linear tensor category of finite-dimensional algebraic representations of $G$.  Recall that
\begin{enumerate}
\item The representation ring of $T$ itself is naturally identified with the Laurent polynomial ring $\bZ[X^*(T)]$, where $X^*(T)$ denotes the character lattice of $T$.  The action of $W$ on $T$ induces an action of $W$ on $\bZ[X^*(T)]$.
\item Restricting to the maximal torus yields an injective homomorphism $\Rep(G) \to \Rep(T)$, which is an isomorphism onto the ring of $W$-invariants $\Rep(T)^W$.
\end{enumerate}

We have the following trivial consequence:

\begin{lemma}
\label{lem:37}
Let $K$ be an algebraically closed field, let $G$ be an reductive algebraic $K$-group and let $H \subset G$ be a reductive algebraic $K$-subgroup of the same rank as $G$.  Let $W_G$ and $W_H$ denote the Weyl groups of $G$ and $H$ respectively.  There is a commutative square
$$\xymatrix{
\Rep(G) \ar[r]^{\mathrm{res}^{G}_{H}} \ar[d]_{\cong} & \Rep(H) \ar[d]^{\cong} \\
\bZ[X^*(T)]^{W_G} \ar[r]_{\mathrm{inclusion}} & \bZ[X^*(T)]^{W_H}
}$$

\end{lemma}

\subsection{Proof of Theorem \ref{thm:centralizer}}
\label{sec:33}

By Remark \ref{rem:autom}, to prove Theorem \ref{thm:centralizer} we only have to verify part (1).  That is, we have to show that given a subgroup $\varpi \subset G$ of order $p$, the group $Z_G(\varpi)^\vee$ injects into $G^\vee$.

Recall some features of Kac's classification of semsimple elements whose centralizer is semisimple: 
\begin{enumerate}
\item Besides the identity element, there is one for each node in the Dynkin diagram associated to $G$.  
\item The Dynkin diagram of the centralizer is obtained by deleting this vertex from the extended Dynkin diagram of $G$.  
\item If $\{\alpha_i\}_{i \in I}$ are the simple roots of $G$ and 
$$\alpha_{\mathrm{top}} = \sum_{i \in I} c_i \alpha_i$$
is the maximal root, then if $c_i = 1$ the semisimple element corresponding to $\alpha_i$ is central, and otherwise it is of order $c_i$.  In the noncentral case the element itself is the image of a $c_i$th root of unity under the fundamental coweight $\beta_i:\Gm \to T$ corresponding to $\alpha_i$.  
\item The center of $Z_G(\varpi)$ is a split extension of $\varpi$ by the center of $G$.
\end{enumerate}

\subsubsection{Type $\mathrm{A}_n$}
The only semisimple elements with semisimple centralizers are in the center of $\SL(n)$, in particular their centralizer is all of $\SL(n)$.  Thus there is nothing to prove.

\subsubsection{Type $\mathrm{B}_n$}
In types $\mathrm{B}_n$ and $\mathrm{C}_n$, the existence of the subgroups has to do with the peculiar nature of quadratic forms in characteristic 2.  The essential fact for $\mathrm{B}_n$ is the following: if $q$ is a nondegenerate quadratic form on an even-dimensional vector space over a field of characteristic 2, then the associated bilinear form is alternating (i.e. we have $B(v,v) =0$).  This gives us an inclusion $\iota_a:\SO(2a) \hookrightarrow \Sp(2a)$.

Let us use this fact to construct the interesting subgroups of $G^\vee$ in characteristic 2.  The group $G$ is $\Spin(2n+1)$, which covers $\SO(2n+1)$.  The Langlands dual is $\Sp(2n)/\mu_2$.  The noncentral semisimple elements with semisimple centralizers are lifts of 
$$
\left(
\begin{array}{cc}
-1_{2a} & 0 \\
0 & 1_{2(n-a)+1}
\end{array}
\right)
$$
where $2 \leq a \leq 2n$.  The centralizer is a double cover of $\SO(2a) \times \SO(2(n-a)+1)$, so that the Langlands dual is a quotient of $\SO(2a) \times \Sp(2(n-a))$ by a diagonal central $\mu_2$.  The inclusion of this into $\Sp(2n)/\mu_2$ is covered by the inclusion
$$\SO(2a) \times \Sp(2(n-a)) \stackrel{\iota_a \times \mathrm{id}}{\longrightarrow} \Sp(2a) \times \Sp(2(n-a)) \to \Sp(2n)$$

\subsubsection{Type $\mathrm{C}_n$} 
\label{subsubsec:Cn}
The group is $\Sp(2n)$, and the interesting centralizers are all of the form $\Sp(2a) \times \Sp(2b)$ where $a+b = n$.  These can be described as the centralizers of an element of order 2, that acts as $-1$ on the $2a$-dimensional symplectic subspaces and $+1$ on the $2b$-dimensional symplectic subspace.  We may describe the Langlands dual inclusion $\SO(2a + 1) \times \SO(2b+1)$ into $\SO(2a+2b+1)$ by noting the following feature of quadratic spaces of odd dimension in characteristic 2.

If $K$ has characteristic 2 and we endow $K^{2a+1}$ and $K^{2b+1}$ with the quadratic forms
$$\begin{array}{c}
x_1x_2 +  \cdots + x_{2a-1} x_{2a} + x_{2a+1}^2 \\
y_1y_2 + \cdots + y_{2b-1} y_{2b} + y_{2b+1}^2
\end{array}
$$
then the line $\ell \subset K^{2a+1} \oplus K^{2b+1}$ given by setting $x_{2a+1} = y_{2b+1}$ and all other coordinates zero has the following properties with respect to the direct sum quadratic form $q$ on $K^{2a+1} \oplus K^{2b+1}$: it is perpendicular to everything, and $q$ vanishes identically on it.  The quotient by this line is equivalent the standard quadratic form on $K^{2a+2b+1}$.  In this way we get a homomorphism $\mathrm{O}(2a+1) \times \mathrm{O}(2b+1) \hookrightarrow \mathrm{O}(2a+2b+1)$ and $\SO(2a+1) \times \SO(2b+1) \hookrightarrow \SO(2a+2b+1)$.

\subsubsection{Type $\mathrm{D}_n$}  The group is $\Spin(2n)$ and the Langlands dual is $\SO(2n)/\mu_2$.  The interesting elements all have order 2:
$$
\left(
\begin{array}{cc}
-1_{2a} & 0 \\
0 & 1_{2(n-a)}
\end{array}
\right)
$$
The centralizers are double covers of $\SO(2a) \times \SO(2(n-a))$, whose Langlands duals are $(\SO(2a) \times \SO(2(n-a)))/\mu_2$.  The inclusion into $\SO(2n)/\mu_2$ is covered by the usual inclusion of $\SO(2a) \times \SO(2(n-a))$ into $\SO(2n)$, which exists in every characteristic.

\subsubsection{Exceptional types}
\label{sec:exc}

$$
\begin{array}{c}

\begin{array}{ccc}
\G_2: \, {\mathbf 2} \equiv \!\!\!\!\! {>}\!\!\!\!\! \equiv {\mathbf 3} & \qquad & \F_4: \, {\mathbf 2} - {\mathbf 3} = \!\!\!\! {>}\!\!\!\! = 4 - {\mathbf 2} 
\end{array} 

\\
\\
\\

\begin{array}{ccc}
\begin{array}{lcccccc}
&  & &  \mathbf{2}  & & &\\
\E_6: &1 & \mathbf{2} & \mathbf{3} & \mathbf{2} & 1
\end{array}
&
\begin{array}{lccccccc}
& & & \mathbf{2}  & & & & \\
\E_7: &\mathbf{2} & \mathbf{3} & 4 & \mathbf{3} & \mathbf{2} & 1
\end{array}
&
\begin{array}{lcccccccc}
& & & \mathbf{3} & & & & &\\
\E_8: & \mathbf{2} & 4 & 6 & \mathbf{5} & 4 & \mathbf{3} & \mathbf{2}
\end{array}
\end{array}
\end{array}
$$
The diagram displays the highest roots in the exceptional root systems, with the prime coefficients in boldface.  We have to investigate subgroups of these groups in characteristics 2, 3, and 5.  Maximal subgroups of the exceptional groups have been classified by Liebeck and Seitz \cite{ting}, and it can be seen by consulting Table 10.3 of their manuscript that, with a single exception, each endoscopic group of $G^\vee$ associated to an element of order $p$ in $G$ does appear as a maximal subgroup of $G^\vee/K$ when $K$ has characteristic $p$.

The exception is the element $x \in \F_4$ of order 2 corresponding to the left-most node of the displayed diagram.  
As the affine Dynkin diagrm of $\F_4$ is
$$\circ - \bullet - \bullet = \!\!\!\! {>}\!\!\!\! =  \bullet - \bullet$$
the centralizer must have Dynkin diagram
$$\circ \qquad \quad \bullet = \!\!\!\! {>}\!\!\!\! =  \bullet - \bullet$$
i.e. it should be a quotient of $\SL(2) \times \Sp(6)$, where the highest root of $\SL(2)$ corresponds to the highest long root of $\F_4$ under the inclusion.  In fact it must be $\SL(2) \times \Sp(6)$ modulo the diagonal copy of $\bZ/2 \cong Z(\SL(2)) \cong Z(\Sp(6))$.  To see this, one can reason as follows: the image of $\Sp(6)$ in $\F_4$ centralizes the image of $\SL(2)$, and in particular it must contain the center of $\SL(2) \to \F_4$ in its center.  This map is injective by the next Lemma, so the image of the map from $\Sp(6)$ does not map $Z(\Sp(6))$ to $1$, so it is itself injective.  

\begin{lemma}
\label{lem:f41}
If $\alpha$ is a root (resp. coroot) of $\F_4$, then $\alpha$ is primitive in the weight lattice, i.e. $\alpha/k$ is not an integral weight for any $k \in \bZ$.  Because of this, the coroot homomorphism $\SL(2) \to \F_4$ is injective in any characteristic.
\end{lemma}

\begin{proof}
As both the center and fundamental group of $\F_4$ are trivial, it suffices to check the first assertion for the roots, and by symmetry it suffices to check it for the simple roots.  In the weight basis, the simple roots of $\F_4$ are the columns of its Cartan matrix
$$
\left(
\begin{array}{rrrr}
2 & -1 & 0 & 0 \\
-1 & 2 & -1 & 0 \\
0 & -2 & 2 & -1 \\
0 & 0 & -1 & 2
\end{array}
\right)
$$
which are evidently all primitive.
\end{proof}

Then $Z_{\F_4}(x)^\vee$ must be $\SL(2) \times \Spin(7)$ mod its diagonal copy of $\bZ/2$.  The coroots of $Z_{\F_4}(x)^\vee$ correspond to the coroots of $\F_4^\vee \cong \F_4$ in the same way that the roots of $Z_{\F_4}(x)$ correspond to the roots of $\F_4$.  If we are to have a map $\SL(2) \times \Spin(7) \to \F_4^\vee$, the $\SL(2)$ factor must map the highest root of $\SL(2)$ to the highest \emph{short} root of $\F_4^\vee$.  Let us denote this root by $\gamma$---its coefficients are given by
$$1 - 2 = \!\!\!\! {>}\!\!\!\! = 3 - 2$$
The previous Lemma shows that the coroot map $\gamma^\vee:\SL(2) \to \F_4$ is injective, even in characteristic 2.

For each root $\alpha$ of $\F_4$, write $U_\alpha:\Ga \hookrightarrow \F_4$ for the corresponding root subgroup.  Let $\mathfrak{u}_\alpha \subset \mathfrak{f}_4$ denote the root space in the Lie algebra of $\F_4$.

\begin{lemma}
\label{lem:f42}
If $\gamma$ denotes the highest short root of the root system $\F_4$,
then the centralizers of $U_\gamma$ and of $\mathfrak{u}_\gamma$ in $\F_4$ coincide in characteristic 2.
\end{lemma}

\begin{proof}
First, one notes that in characteristic 2 the unipotent radical $U$ of the Borel subgroup $B$ of $\F_4$ is contained in both centralizers.  For $\mathfrak{u}_\gamma$, this follows from \cite[Table 1]{spalt}.  For $U_\gamma$ one can consult the commutation relations in \cite[Section 2]{sh}.  Let $T'$ denote the kernel of $\gamma$ regarded as a weight $T \to \Gm$.  $T'$ and $U$ together generate the centralizer of $U_\gamma$ (resp. of $\mathfrak{u}_\gamma$) in $B$.

Let $W$ denote the Weyl group of $\F_4$, and for each $w \in W$ fix an element $n_w \in N(T)$ mapping to $W$, such that $n_w U_\alpha(a) n_w^{-1} = U_{w \alpha}(a)$.  Using the Bruhat decomposition we can write a general element of $\F_4$ as $g = t u_1 n_w u_2$ where $t \in T$ and $u_1,u_2 \in U$.  As both $u_1$ and $u_2$ are automatically in the centralizer of $U_\gamma$ (resp. $\mathfrak{u}_\gamma$), we see that $g$ is in the centralizer if and only if $t \in T'$ and $n_w$ commutes with $U_\gamma$ (resp. with $\mathfrak{u}_\gamma$).  This holds if and only if $w$ stabilizes $\gamma$, which holds if and only if $w$ is generated by the simple reflections associated to the left three nodes.  In particular $g$ belongs to the centralizer of $U_\gamma$ if and only if it belongs to the centralizer of $\mathfrak{u}_\gamma$.
\end{proof}

The centralizer of $\SL(2) \to \F_4$ is the reductive part of the centralizer of $U_\gamma$, which by the Lemma coincides with the reductive part of the centralizer of $\mathfrak{u}_\gamma$.  By \cite[Table 1]{spalt}, this is a simple algebraic group of type $\mathrm{B}_3$, i.e. it is either $\Spin(7)$ or $\SO(7)$.  By Lemma \ref{lem:f41}, the center of this centralizer must contain the center of $\SL(2)$, so we see that it is $\Spin(7)$.  This produces a map $\SL(2) \times \Spin(7) \to \F_4$ whose kernel is the diagonal copy of $\mu_2 = Z(\SL(2)) = Z(\Spin(7))$, as required.

\begin{remark}
By inspecting the table in \cite{ting}, one sees that in these exceptional types there are only three endoscopic groups that are not subgroups in all characteristics: $\PGL(3)$ for $\G_2$, correspoding to the node labeled ``3'' in the $\G_2$ diagram. $\Sp(8)$ for $\F_4$ corresponding to the right-most node in the $\F_4$ diagram, and the $\SL(2)  \times \Spin(7) /\mu_2$ in $\F_4$ that we have just discussed, corresponding to the left-mode node in the $\F_4$ diagram.
\end{remark}

\section{Smith theory for sheaves}

Let $X$ be a real subanalytic or complex algebraic variety.  Let $K$ be a commutative ring.  We let $D^b_{\bR-c}(X;K)$ (resp. $D^b_{\bC-c}(X;K)$) denote the triangulated category of bounded cohomologically $\bR$-constructible (resp. $\bC$-constructible) sheaves of $K$-modules on $X$.  
We will usually abuse notation and write $D^b_c(X;K)$ for one of these categories, and it should be clear from context whether we are in the subanalytic or complex algebraic settings.  If $G$ is a Lie group (resp. complex algebraic group) acting subanalytically (resp. algebraically) on $X$, write $D^b_G(X;K)$ for the $G$-equivariant version of this category.

\subsection{The Tate coefficient category}

Let $\varpi = \bZ/p$ and let $K$ be a field of characteristic $p$.  Let $K[\varpi]$ be the group ring of $\varpi$, and let $D^b(K[\varpi])$ be the bounded derived category of finitely-generated $K[\varpi]$-modules.  We have a thick subcategory $\Perf(K[\varpi])$ spanned by bounded complexes of free $K[\varpi]$-modules.

\begin{definition}
The \emph{Tate category} is the Verdier quotient category $D^b(K[\varpi])/\Perf(K[\varpi])$.  Write $\Perf(\tate)$ for the Tate category.
\end{definition}

\begin{proposition}
The Grothendieck group of the Tate category is $\bZ/p$, generated by the class of the trivial $K[\varpi]$-module $K$.
\end{proposition}

\begin{proof}
The algebra $K[\bZ/p]$ is local: it has only one simple module $K$.  It follows that $K$ generates the Grothendieck group of $D^b(K[\varpi])$ as well as the Grothendieck group of any localization of $D^b(K[\varpi])$.  To show that the Grothendieck group of $\Perf(\tate)$ is $\bZ/p$ then, it suffices to exhibit a $\bZ/p$-valued invariant $\chi$ of objects of $D^b(K[\varpi])$ with the following properties:
\begin{enumerate}
\item $\chi$ is additive for exact triangles
\item $\chi(M^\bullet) = 0$ when $M^\bullet$ belongs to $\Perf(K[\varpi])$
\item $\chi(K) = 1$
\end{enumerate}
It is easy to check that the invariant
$$\chi(M^\bullet) = \sum_{i \in \bZ} (-1)^i \dim_K(M^i) \text{ mod  }p$$
has the required properties.
\end{proof}

\begin{proposition}
\label{prop:periodic}
The shift-by-2 functor $\Perf(\tate) \to \Perf(\tate): M \mapsto M[2]$ is naturally isomorphic to the identity functor.  If $p = 2$, then the shift-by-1 functor is naturally isomorphic to the identity functor.
\end{proposition}

\begin{proof}
Let $g$ be a generator of $\varpi$.  For any $K[\varpi]$-module $M$, we have the 
exact sequence
$$0 \to M \to M \otimes_K K[\varpi] \stackrel{1 - g}{\to} M \otimes_K K[\varpi] \to M \to 0$$
where $\varpi$ acts diagonally on the middle two terms.  The associated short exact sequence of cochain complexes
$$\xymatrix{
0 \ar[r] & M \ar[r] \ar[d] & M \otimes_K K[\varpi] \ar[r] \ar[d] & 0 \ar[r]  \ar[d] & 0 \\
0 \ar[r] & 0 \ar[r] & M \otimes_K K[\varpi] \ar[r] & M \ar[r] & 0
}
$$
induces a map $M \to M[2]$ whose cone is a 2-term complex of free modules.  The proposition follows.

When $p = 2$ we may use the shorter exact sequence
$$0 \to M \to M \otimes_K K[\varpi] \to M \to 0$$
to deduce the Proposition.
\end{proof}

\begin{remark}
\label{rem:tatespec}
The proposition shows that $\Perf(\tate)$ is not the derived category of any abelian category, and indeed can carry no $t$-structure at all.  However it can be shown that $\Perf(\tate)$ is equivalent to the homotopy category of a certain category of  module spectra over an $E_\infty$-ring spectrum $\tate$.  Basically, $\tate$ is the natural ring spectrum whose homotopy groups are the Tate cohomology groups of $\bZ/p$ with coefficients in $K$.  If $p$ is odd, then the homotopy groups of $\tate$ are
$$
\begin{array}{c}
\pi_{2i}(\tate) = K \text{ with generator }x^i\\
\pi_{2i+1}(\tate) = K \text{ with generator } x^i y
\end{array}
$$
with the evident ring structure.  If $p=2$, then we have $\pi_\bullet(\tate) = K[y,y^{-1}]$ with $y \in \pi_1$.  The fact that the natural class in $\pi_2$ is invertible accounts for Proposition \ref{prop:periodic}.
\end{remark}

\subsubsection{Tensor structure}
\label{sec:tensorcoef}

Related to Remark \ref{rem:tatespec}, it is possible to endow $\Perf(\tate)$ with a symmetric monoidal structure.  If $M^\bullet$ and $N^\bullet$ are two bounded complexes of finitely generated $K[\varpi]$-modules, the tensor product $M^\bullet \otimes_K N^\bullet$ is another bounded complex equipped with the diagonal $K[\varpi]$-module structure, endowing $D^b(K[\varpi])$ with a symmetric monoidal structure.  

 If $M^\bullet \in D^b(K[\varpi])$ and $N^\bullet \in \Perf(K[\varpi])$, then $M^\bullet \otimes_K N^\bullet \in \Perf(K[\varpi])$.  Thus, $\otimes_K$ descends to a symmetric monoidal structure on $\Perf(\tate)$, which we denote by $\otimes_{\tate}$.

\begin{remark}
\label{rem:notKpilinear}
Since $K[\varpi]$ is a commutative ring, we can define the $K[\varpi]$-linear tensor product $\stackrel{\mathbf{L}}{\otimes}_{K[\varpi]}$.  Note that this is not the one we are considering when we define $\otimes_{\tate}$.  Indeed, $D^b(K[\varpi])$ is not even closed under $\stackrel{\mathbf{L}}{\otimes}_{K[\varpi]}$, e.g. $K \stackrel{\mathbf{L}}{\otimes}_{K[\varpi]} K$ is unbounded below.
\end{remark}

\subsubsection{Duality}
\label{sec:dualitycoef}

If $M^\bullet$ is a bounded complex of finitely generated $K[\varpi]$-modules, we let $(M^\bullet)^*$ denote the complex of dual $K$-vector spaces equipped with the contragredient $\varpi$-action.  As in Remark \ref{rem:notKpilinear}, we note that this functor differs from the duality functor $M^\bullet \mapsto \mathbf{R}Hom_{K[\varpi]}(M^\bullet,K[\varpi])$.

If $M^\bullet$ is a bounded complex of free modules then so is $(M^\bullet)^*$.  The functor $M \mapsto M^*$ therefore descends to a duality functor on $\Perf(\tate)$, which we denote by $\bD$.

\subsection{Tate coefficients and the Smith operation}

If $Y$ is a real subanalytic variety and $\varpi$ acts trivially on $Y$, then we make the identification
$$D^b_\varpi(Y;K) \cong D^b(Y;K[\varpi])$$
Let us denote by $\Perf(Y;K[\varpi]) \subset D^b(Y;K[\varpi])$ the full subcategory spanned by sheaves of $K[\varpi]$-modules all of whose stalks are perfect.  We will denote the Verdier quotient of $D^b(Y;K[\varpi])$ by $\Perf(Y;K[\varpi])$ by $\Perf(Y;\tate)$.

\begin{lemma}
\label{lem:link}
Let $X$ be a finite-dimensional space on which $\varpi$ acts freely.  Then the global sections functor $\Gamma:D^b_\varpi(X;K) \to D^b_\varpi(\mathit{pt};K) = D^b(K[\varpi])$ takes values in $\Perf(K[\varpi])$.
\end{lemma}

\begin{proof}
It suffices to show that $\Gamma(F)$ is a perfect complex of $K[\varpi]$-modules when $F$ is the constant sheaf on a $\varpi$-invariant closed subset $Y$, as these sheaves generate $D^b_\varpi(X;K)$.  Pick a $\varpi$-invariant triangulation of $Y$.  Then $\Gamma(F)$ is quasi-isomorphic to the simplicial cochain complex of this simplicial with coefficients in $K$, together with its natural $\varpi$-action.  As $\varpi$ acts freely on $Y$ it acts freely on the set of $i$-simplices in $Y$, and therefore this cochain complex is perfect.
\end{proof}

\begin{theorem}
\label{thm:link}
Let $X$ be a $\varpi$-space and let $i$ denote the inclusion of $X^\varpi$ into $X$.  The cone on the natural map $i^! \to i^*$ belongs to $\Perf(X^\varpi;K[\varpi])$.
\end{theorem}

\begin{proof}
Let $x$ be a $\varpi$-fixed point of $X$, and let $U$ be a regular neighborhood of $x$.  As $\varpi$ is finite we may assume $U$ is $\varpi$-invariant.  Let $L = U - (U \cap X^\varpi)$.  A standard computation identifies the stalk of $C$ at $x$ with the cohomology of $L$ with coefficients in $F\vert_L$.  By Lemma \ref{lem:link} this is perfect.  It follows that $C$ is perfect.
\end{proof}

\begin{definition}
The sheaf-theoretic \emph{Smith operation} is the composite functor
$$D^b_\varpi(X;K) \stackrel{i^*}{\to} D^b_\varpi(X^\varpi,K) \cong D^b(X^\varpi,K[\varpi]) \to \Perf(X^\varpi;\tate)$$
We denote the functor by $\Sm$.
\end{definition}

\begin{remark}
The previous theorem shows that we could define this operation with $i^!$ in place of $i^*$.
\end{remark}

\begin{remark}
\label{rem:hyperbolic}
If $X$ is a complex algebraic variety carrying an action of $\bC^*$, then in between $X$ and $X^{\bC^*}$ we have the attracting set $X^+$.  The hyperbolic localization functor is defined to be the composition of shriek and star restriction functors
$$(X^{\bC^*} \hookrightarrow X^+)^! \circ (X^+ \hookrightarrow X^{\bC^*})^*$$
Smith localization is analogous to hyperbolic localization in the following sense: instead of combining the two restriction functors in a clever way, we simply erase the distinction between them.
\end{remark}

\subsection{Six operations with Tate coefficients}

Suppose that $Y$ is a variety equipped with the trivial $\varpi$-action.

\subsubsection{Duality and tensor product}

Under the identification $D^b_\varpi(Y;K) \cong D^b(Y;K[\varpi])$, the $\varpi$-equivariant Verdier duality operation is a sheaf version of the operation considered in section \ref{sec:dualitycoef}.  Since an object of $D^b(Y;K[\varpi])$ belongs to the subcategory $\Perf(Y;K[\varpi])$ if and only if each stalk belongs to $\Perf(K[\varpi])$, the duality operation preserves $\Perf$ and descends to an operation on $\Perf(Y;\tate)$.  Similarly the tensor product considered in section \ref{sec:tensorcoef} gives a symmetric monoidal structure on $D^b(Y;K[\varpi])$ that descends to a symmetric monoidal structure on $\Perf(Y;\tate)$.

\subsubsection{Pushforward and pullback}
\label{subsubsec:pandp}

Let $Y'$ be a second variety equipped with the trivial $\varpi$-action, and let $u:Y \to Y'$ be a morphism.  If $F'$ is a sheaf of $K[\varpi]$-modules on $Y'$, then the stalk of $u^* F'$ at $y$ is isomorphic to the stalk of $F'$ at $u(y)$.  It follows that $u^*$ carries perfect sheaves of $K[\varpi]$-modules to perfect sheaves of $K[\varpi]$-modules, and descends to an operation $u^*:\Perf(Y';\tate) \to \Perf(Y;\tate)$.  Similary:

\begin{proposition}
If $F$ is a sheaf of perfect $K[\varpi]$-modules on $Y$ then $u_! F$ is a sheaf of perfect $K[\varpi]$-modules on $Y'$.
\end{proposition}

\begin{proof}
By proper base-change we may assume $Y'$ is a point.  By induction on the length of $F$ we may furthermore assume assume that $F$ is a sheaf of free $K[\varpi]$-modules concentrated in a single degree, i.e. $F = F_1 \otimes_K K[\varpi]$ where $F_1$ is a sheaf of $K$-modules concentrated in a single degree.  Then $u_!(F) = u_!(F_1) \otimes_K K[\varpi]$ is perfect since $u_!(F_1)$ vanishes in degrees $\geq \dim(Y)$.
\end{proof}

It follows that $u_!$ induces a functor $\Perf(Y;\tate) \to \Perf(Y';\tate)$.  Since Verdier duality preserves perfect sheaves we also have well-defined functors $u_*$ and $u^!$.

\subsection{Compatibility of Smith with six operations}

Two classical applications of Smith theory are the following:

\begin{enumerate}
\item In \cite{Q}, Quillen extends Smith's original result (Theorem \ref{thm:1.2}), and shows that the cohomology of a finite-dimensional space with mod $p$ coefficients is closely related to the cohomology of the $\bZ/p$-fixed points with mod $p$ coefficients.
\item In \cite{Bre} and \cite{CK}, it is shown that the fixed points of a $\bZ/p$-action on a space that satisfies Poincar\'e duality mod $p$ again satisfies Poincar\'e duality mod $p$.
\end{enumerate}

These results are consequences of the following general principle: the Smith operation commutes with all other operations.  A generalization of (1) states that $\Sm$ is compatible with pushforwards, and a generalization of (2) states that $\Sm$ is compatible with Verdier duality.  A somewhat more trivial result is that $\Sm$ commutes with pullback; let us prove this result first.

\begin{theorem}
\label{thm:Smithpull412}
Let $X$ and $Y$ be real subanalytic varieties with an action of $\varpi$.  Let $f:X \to Y$ be a $\varpi$-equivariant morphism between them.  The square
$$
\xymatrix{
D^b_\varpi(Y;K) \ar[r]^{f^*} \ar[d]_{\Sm} & D^b_\varpi(X;K) \ar[d]^{\Sm} \\
\Perf(Y^\varpi;\tate) \ar[r]_{f^*} & \Perf(X^\varpi;\tate) }
$$
commutes up to a natural isomorphism.
\end{theorem}

\begin{proof}
Let $i_X:X^\varpi \to X$ and $i_Y:Y^\varpi \to Y$ denote the inclusion maps.    
As $f \circ i_X = i_Y \circ (f\vert_{X^\varpi})$ we have a natural isomorphism $i_X^* \circ f^* F \cong (f\vert_{X^\varpi})^*\circ i_Y^* F$ in $D^b_\varpi(X^\varpi;K)$, which induces an isomorphism between $\Sm \circ f^* F$ and $(f\vert_{X^\varpi})^* \circ \Sm(F)$ in $\Perf(X^\varpi;\tate)$.

\end{proof}

\begin{theorem}
\label{thm:duality}
Let $X$ be a real subanalytic variety with an action of $\varpi$.  The square
$$
\xymatrix{
D^b_\varpi(X;K) \ar[r]^{\bD_X} \ar[d]_{\Sm} & D^b_\varpi(X;K) \ar[d]^{\Sm} \\
\Perf(X^\varpi;\tate) \ar[r]_{\bD_{X^\varpi}} &  \Perf(X^\varpi;\tate)
}
$$
commutes up to a natural isomorphism.
\end{theorem}

\begin{proof}
Let $i:X^\varpi \to X$ denote the inclusion.  We construct a natural transformation $\bD i^* F \to i^* \bD F$ by composing the natural isomorphism $\bD i^* F \cong i^! \bD F$ with the natural transformation $i^! \bD F \to i^* \bD F$.  To show that the induced map $\Sm \bD F \to \bD \Sm(F)$ is an isomorphism it suffices to show that the cone on $\bD i^* F \to i^* \bD F$ is perfect---this follows from Theorem \ref{thm:link}.
\end{proof}

\begin{theorem}
\label{thm:Smithpush}
Let $X$ and $Y$ be real subanalytic varieties with an action of $\varpi$.  Let $f:X \to Y$ be a $\varpi$-equivariant morphism between them.  The squares
$$
\xymatrix{
D^b_\varpi(X;K) \ar[r]^{f_!} \ar[d]_{\Sm} & D^b_\varpi(Y;K) \ar[d]^{\Sm} & & D^b_\varpi(X;K) \ar[r]^{f_*} \ar[d]_{\Sm} & D^b_\varpi(Y;K) \ar[d]^{\Sm} \\
\Perf(X^\varpi;\tate) \ar[r]_{(f\vert_{X^\varpi})_!} & \Perf(Y^\varpi;\tate) & & \Perf(X^\varpi;\tate) \ar[r]_{(f\vert_{X^\varpi})_*}&  \Perf(Y^\varpi;\tate)
}
$$
commute up to natural isomorphisms.
\end{theorem}

\begin{proof}
As $\bD \circ f_! \circ \bD = f_*$, it is enough to consider the left-hand square.  Let $i_X:X^\varpi \to X$ and $i_Y:Y^\varpi \to Y$ be the inclusion maps.  We have an adjunction morphism
$$i_Y^* f_! \to (f\vert_{X^\varpi})_! i_X^*$$
in the category of functors from $D^b_\varpi(X;K) \to D^b_\varpi(Y^\varpi;K)$, which induces a morphism $n:\Sm \circ f_! \to f_! \circ \Sm$.  

To show that $n$ is an isomorphism it suffices to show that the cone on $i_Y^* f_! F \to (f\vert_{X^\varpi})_! i_X^*F$ is a perfect.  Since this may be checked on stalks, we may as well assume that $Y$ is a single point.  We may furthermore reduce to the case where $F$ is a constant sheaf on a closed subanalytic $\varpi$-invariant subset, as these sheaves generate $D^b_\varpi(X;K)$.  Now we only have to check that the cone on the map
$$\Gamma_c(K_X) \to \Gamma_c(K_{X^\varpi})$$
has a perfect cone.  This may be verified as in the proof of Lemma \ref{lem:link}.
\end{proof}

\begin{theorem}
\label{thm:Smithnearby}
Let $X$ be a complex algebraic variety with an action of $\varpi$, and let $f:X \to \bC$ be a $\varpi$-invariant map, and let $\psi_f$ denote the nearby cycles functor
$$\psi_f:D^b_\varpi(X;K) \to D^b_\varpi(f^{-1}(0);K)$$
The square
$$\xymatrix{
D^b_\varpi(X;K) \ar[r]^{\psi_f\quad} \ar[d]_{\Sm} & D^b_\varpi(f^{-1}(0);K) \ar[d]^{\Sm} \\
\Perf(X^\varpi;\tate) \ar[r]_{\psi_f\quad} & \Perf(f^{-1}(0)^\varpi;\tate)
}
$$
commutes up to natural isomorphism.
\end{theorem}

\begin{proof}
We have $\psi_f = i^*j_* j^*$ where
$j$ is the inclusion of $f^{-1}(\bR_{>0})$ into $X$ and
$i$ is the inclusion of $f^{-1}(0)$ into $X$.  The Theorem the follows from Theorems \ref{thm:Smithpull412} and \ref{thm:Smithpush}.

\end{proof}

\subsection{Smith theory for equivariant sheaves}

Let $X$ be a subanalytic space and $G$ a Lie group acting subanalytically on $X$.  The bar construction $[X/G]$ is the simplicial space whose $n$th term is $G^{\times n} \times X$, and whose face and degeneracy maps are given by multiplication and insertion.  A $G$-equivariant sheaf on $X$ is a complex of simplicial sheaves on $[X/G]$ whose cohomology simplicial sheaves are ``effective'' or ``Cartesian.''  We refer to \cite{BL} for more details; informally we are given a complex of sheaves $F^k$ on each $G^{\times k} \times X$, together with a quasi-isomorphism $\phi^* F^\ell \to F^k$ for each structure map $\phi:G^{\times k} \times X \to G^{\times \ell} \times X$, and these quasi-isomorphisms are required to be compatible with each other in a suitable sense.  We call an equivariant sheaf bounded if each $F^k$ is finitely many nonzero cohomology sheaves, or equivalently if $F^0$ has finitely many nonzero cohomology sheaves.

Let us denote the triangulated category of bounded $G$-equivariant sheaves on $X$ by $D^b([X/G];K)$.  If we have a subgroup $\varpi \subset G$, then $\varpi$ acts on each term $G^{\times k} \times X$ of $[X/G]$ in the following way:
$$h \cdot (g_1,\ldots,g_k,x) = (hg_1h^{-1},\ldots,hg_k h^{-1},hx)$$
We let $D^b_\varpi([X/G];K)$ denote the category of Cartesian sheaves on the bisimplicial space $\varpi^j \times G^k \times X$.  Roughly speaking, to give an object of $D^b_\varpi([X/G];K)$ one gives a $\varpi$-equivariant sheaf $F^k$ on each $G^{\times k} \times X$, together with a $\varpi$-equivariant quasi-isomorphism for each structure map $G^{\times k} \times X \to G^{\times \ell} \times X$.

\begin{definition}
Applying $\Sm$ term-by-term gives us a functor 
$$D^b_\varpi([X/G];K) \to D^b([X^\varpi/Z_G(\varpi)];\tate)$$
We denote this functor by $\Sm'$.
\end{definition}

A standard argument shows that the category of Cartesian bisimplicial sheaves on a bisimplicial space is equivalent to the category of Cartesian simplicial sheaves on the diagonal.  In the case of $\varpi^{\times j} \times G^{\times k} \times X$ we may describe this diagonal concretely:

\begin{proposition}
\begin{enumerate}
\item The simplicial space $\varpi^{\times n} \times G^{\times n} \times X$ is naturally isomorphic to the bar construction $[X/(G \rtimes \varpi)]$, where the semidirect product $G \rtimes \varpi$ is constructed using the conjugation action of $\varpi$ on $G$, and $G \rtimes \varpi$  acts on $X$ via $((g,h),x) \mapsto ghx$.
\item The map $G \rtimes \varpi \to G \times \varpi:(g,h) \mapsto (gh,h)$ is a group isomorphism.
\end{enumerate}
\end{proposition}

It follows from (2) that we have a simplicial map $f:[X/(G \rtimes \varpi)] \to [X/G]$.  We define the Smith operator for equivariant sheaves as the composite

$$D^b([X/G];K) \stackrel{f^*}{\to} D^b([X/G \rtimes \varpi];K) \cong D^b_\varpi([X/G];K) \stackrel{\Sm'}{\longrightarrow} D^b([X^\varpi/Z_G(\varpi)];\tate)$$

\subsection{Conjecture on perverse sheaves}
\label{sec:conjecture}

By Proposition \ref{prop:periodic}, the category $\Perf(Y;\tate)$ can carry no t-structure.  Nevertheless, I believe that the Smith operator 
$$\Sm:D^b_\varpi(X;K) \to \Perf(X^\varpi;\tate)$$
interacts well with the perverse $t$-structure on $D^b_\varpi(X;K)$ when $X$ is a complex algebraic variety.  Before stating the conjecture, note that there is a natural functor
$$D^b(X^\varpi;K) \to \Perf(X^\varpi;\tate)$$
that carries a sheaf $F$ first to $F \otimes_K K[\varpi]$ in $D^b_\varpi(X^\varpi;\tate)$ and then to its image in $\Perf(X^\varpi;\tate)$.  Let us denote this functor by $\otimes_K \tate$, as suggested by Remark \ref{rem:tatespec}.

\begin{conjecture}
Let $X$ be a complex algebraic variety equipped with a $\varpi$-action, and let $P$ be a $\varpi$-equivariant perverse sheaf of $K$-vector spaces on $X$.  Then there exist perverse sheaves of $K$-vector spaces $P_1, \ldots, P_n$ on $X^\varpi$ and integers $a_1,\ldots,a_n$ such that 
$$\Sm(P) \cong (P_1[a_1] \oplus \cdots \oplus P_n[a_n]) \otimes_K \tate$$
in $\Perf(X^\varpi;\tate)$
\end{conjecture}

I furthermore believe that the $P_i$ are at least some of the time functorially associated to $P$, but I do not know how to formulate this precisely.  Note that by Proposition \ref{prop:periodic}, the $a_i$ are irrelevant for $p = 2$ and only relevant mod $2$ for $p > 2$.

Let me discuss some evidence for this conjecture
\begin{enumerate}
\item Let $U$ be a smooth affine open subset of $X$ and let $i:U \to X$ denote the inclusion map---then $K_U[\dim(U)]$ is a perverse sheaf on $U$ and a theorem of Artin shows that $i_! K_U[\dim(U)]$ and $i_* K_U[\dim(U)]$ are perverse on $X$.  If $U$ is stable for the $\varpi$-action then these are $\varpi$-equivariant, and the space of fixed points $U^\varpi$ is again a smooth affine open subset of $X^\varpi$, and $i_! K_{U^\varpi}$ and $i_*K_{U^\varpi}$ are also shifts of perverse sheaves.  The conjecture then holds for this class of perverse sheaves by Theorem \ref{thm:Smithpush}.

\item A similar argument shows that the Conjecture holds for perverse sheaves of the form $\psi_f K_{f^{-1}(1)}$, where $f:Y \to \bC$ is a family of varieties whose general fiber is smooth.

\item Microlocal considerations can be used to justify the Conjecture.  Nadler and Zaslow \cite{NZ}
construct a dictionary between Lagrangian submanifolds of $T^*M$ and constructible sheaves on $M$, lifting the less functorial, many-to-one dictionary  between constructible sheaves and conic Lagrangian subsets considered by Kashiwara and Schapira \cite{KS}.  When $M$ is a complex manifold it is natural to ask which Lagrangians in $T^*M$ correspond to perverse sheaves on $M$.   The answer, up to shift, is those Lagrangians which are also complex submanifolds (possibly immersed) in the natural complex structure on $T^*M$---this is an unpublished result of Nadler's.  

Now one can reason as follows: we may replace the perverse sheaf $P$ by a complex Lagrangian $L \subset T^* M$.  We may hope to express the $\varpi$-equivariance of $P$ by saying that $L$ is stable for the $\varpi$-action on $T^*M$, and that the Smith operator should carry $L$ to $L^\varpi$ (Section \ref{sec:singsupp} gives some evidence for this idea).  If $L$ is a complex submanifold of $T^* M$ then $L^\varpi$ will be a complex submanifold of $T^* (M^\varpi)$, which is consistent with the Conjecture.

\item We argued in Remark \ref{rem:hyperbolic} that the Smith operator is analogous to hyperbolic localization for $T$-actions.  It is proved in \cite{Braden} that hyperbolic localization interacts well with perverse sheaves.
\end{enumerate}

\subsubsection{Example}

Let $X$ be the one-point compactification of the total space of the line bundle $\cO(n)$ on $\bP^1$---denote the cone point by $\ast$.  If $n$ is positive, this is an algebraic variety.  A basic example of a perverse sheaf on $X$ is the intersection homology sheaf $IC_X = j_{!*}K[2]$ of Deligne-Goresky-MacPherson.  At a smooth point of $X$, the stalk of $IC_X$ is $K[2]$, while at the point $\ast$ the stalk cohomology is ``half the cohomology of the link,'' i.e. it is $H^0(L;K)$ in degree $-2$ and $H^1(L;K)$ in degree $-1$ and zero in other degrees, where $L$ denotes the 3-dimensional link of $\ast$ in $X$.  

In fact $L$ is the lens space $(S^3 - \{0,0\})/\mu_n$.  If $p$ does not divide $n$ then $H^1(L;K) = 0$, so that $IC_X$ is just the shifted constant sheaf.  Let us instead suppose that $p$ does divide $n$, so that $H^1(L,K) = K$. 

Let $\bZ/p$ act on $\bC^2$ by $(x,y) \mapsto (x,\eta \cdot y)$, where $\eta$ is a $p$th root of unity.  This induces an action of $\bZ/p$ on $X = \ast \cup (\bC^2 - \{0\})/\mu_n \cup \bP^1$ whose fixed point set is the one-point compactification of the two lines $\cO(n)_0 \coprod \cO(n)_\infty$---topologically, this is a wedge of two $\bP^1$s, call them $Y_1$ and $Y_2$.

On each $Y_i$, there is up to isomorphism a unique indecomposable perverse sheaf $P_i$ that is constant on $Y_i - \ast$, that is isomorphic to its Verdier dual, and that admits a surjection (in the perverse $t$-structure) onto the skyscraper sheaf at $\ast$.  (It can be described as the projective cover as well as the injective hull of this skyscraper sheaf, and also as the tilting extension of the constant perverse sheaf on $Y_i - \ast$).

Since $IC_X$ is the constant sheaf $K[2]$ along $Y_i - \ast$ for $i=1,2$, we have a map
$$IC_X \vert_{Y_1 \cup Y_2} \to f_{1*} K[2] \oplus f_{2*} K[2]$$
where $f_i$ denotes the inclusion of $Y_i - \ast$ into $Y_i$.  The cone on this map is a skyscraper sheaf supported on $\ast$ placed in degree $-1$, denote it by $\delta[1]$.

Note that $f_1$ and $f_2$ are affine, so that $f_{i*} K[1]$ is perverse.  Thus $\Sm(IC_X)$ is isomorphic to $P[1] \otimes_K \cT$, where $P$ is the kernel of the surjective map of perverse sheaves $f_{1*} K[1] \oplus f_{2*} K[1] \to \delta[1]$
\newline

\noindent
\emph{Acknowledgments:} I thank Florian Herzig, Gopal Prasad, and Ting Xue for help with algebraic groups.  In particular, most of the material in Section \ref{sec:exc} I learned from Ting.  While developing these ideas I benefited from discussions with Paul Goerss, David Nadler, and Zhiwei Yun.

\end{document}